\documentclass[12pt,fleqn,leqno]{amsart}

\usepackage{microtype}
\usepackage{amsfonts,amssymb}
\usepackage{enumitem}
\usepackage{thmtools}
\usepackage[svgnames]{xcolor}
\usepackage[colorlinks,linkcolor=blue,citecolor=Green]{hyperref}
\usepackage{titlesec}
\usepackage{fourier}
\usepackage[nameinlink]{cleveref}
\usepackage{tikz-cd}
\usetikzlibrary{decorations.pathmorphing}
\usepackage{wrapfig}        
\usepackage{nccmath}

\usepackage[a4paper,centering]{geometry}

\titleformat{\section}[block]
 {\bfseries}
 {\thesection.}
 {\fontdimen2\font}
 {}

\setlist{noitemsep}
\setenumerate{labelindent=\parindent,label=\upshape{(\alph*)}}

\newtheorem{theorem}{Theorem}[section]
\newtheorem{corollary}[theorem]{Corollary}
\newtheorem{proposition}[theorem]{Proposition}
\newtheorem{lemma}[theorem]{Lemma}

\theoremstyle{definition}
\newtheorem{remark}[theorem]{Remark}
\newtheorem{example}[theorem]{Example}
\newtheorem{problem}{Problem}

\DeclareMathOperator{\N}{\mathbb{N}}
\DeclareMathOperator{\R}{\mathbb{R}}
\DeclareMathOperator{\s}{\mathbb{S}}
\DeclareMathOperator{\uhr}{\upharpoonright}
\DeclareMathOperator{\conv}{conv}
\DeclareMathOperator{\Z}{\mathbb{Z}}
\DeclareMathOperator{\lip}{Lip}
\DeclareMathOperator{\diam}{diam}
\DeclareMathOperator{\dom}{dom}

\DeclareFontFamily{OT1}{pzc}{}
\DeclareFontShape{OT1}{pzc}{m}{it}{<-> s * [1.150] pzcmi7t}{}
\DeclareMathAlphabet{\mathpzc}{OT1}{pzc}{m}{it}

\numberwithin{equation}{section}
\allowdisplaybreaks

\begin{document}

\author{Valentin Gutev}

\address{Institute of Mathematics and Informatics, Bulgarian Academy
  of Sciences, Acad. G. Bonchev Street, Block 8, 1113 Sofia, Bulgaria}

\email{\href{mailto:gutev@math.bas.bg}{gutev@math.bas.bg}}



\title{Kirszbraun's Intersection Property and Nonexpansive Extensions}
\begin{abstract}
  This paper aims to present the historical development of the
  extension problem of nonexpansive maps in Euclidean and Hilbert
  spaces, and to serve as a comprehensive guide to this specific
  area. Such extensions are based on a geometric property of the
  intersection of closed balls in Euclidean spaces, known as the
  Kirszbraun intersection property. Subsequently, a number of authors
  have extended this property to Hilbert and other spaces. All these
  results are discussed in this paper, but the paper itself is not a
  mechanical compilation of already known proofs. The material has
  been developed with an emphasis on the natural connections between
  all these results and the corresponding concepts involved in them.
  Furthermore, most of these results are presented with new simplified
  proofs. Last but not least, the paper also contains some new results
  related to this extension problem.
\end{abstract}

\dedicatory{Dedicated to Professor Gerald Beer on the occasion of his
  80th birthday}

\date{\today}
\maketitle

\section{Introduction}

In this paper, all vector spaces are real vector spaces. In the
special case of the $n$-dimensional vector space $\R^n$, each element
$(x_1,\dots,x_n)\in \R^n$ will be denoted by
${\mathbf{x}=(x_1,\dots,x_n)}$, so that it carries the same letter as
its coordinates.  The space $\R^n$ equipped with any norm is often
called a \emph{Minkowski space}. In case $\R^n$ is endowed with the
\emph{Euclidean norm} $\|\mathbf{x}\|=\sqrt{x_1^2+\dots +x_n^2}$, we
refer to it as the \emph{$n$-dimensional Euclidean space}, or simply
as an \emph{Euclidean space}. In all our considerations, unless
otherwise stated, it is assumed that $\mathbb{R}^n$ is equipped with
the Euclidean norm.\medskip

For $K\geq 0$ and metric spaces $(X,d)$ and $(Y,\rho)$, a map
$f : X\to Y$ is \emph{$K$-Lipschitz} if $\rho(f(p),f(q))\leq K d(p,q)$
for every $p,q\in X$. Any such $K\geq0$ is referred to as a
\emph{Lipschitz constant} for the Lipschitz map $f$. The smallest
Lipschitz constant $K\geq0$ is sometimes called the \emph{best}
Lipschitz constant of $f$ and denoted by $\lip(f)$. More generally, to
each map $f:X\to Y$ we may associate the (possibly infinite) number
$\lip(f)$ defined by
\[
  \lip(f)=\sup_{p\neq q}\frac{\rho(f(p),f(q))}{d(p,q)}.
\]
This number can be regarded as the \emph{general} Lipschitz constant
of $f$. Namely, the case $\lip(f) =+\infty$ indicates that $f$ is not
Lipschitz, and $\lip(f)=0$ indicates that $f$ is constant. Thus, the
essential case of Lipschitz maps is when $0<\lip(f)<+\infty$.\medskip

The Lipschitz extension problem deals with conditions on the metric
spaces $(X,d)$ and $(Y,\rho)$ such that every $K$-Lipschitz map
$\varphi:A\to Y$ on a subset $A\subseteq X$ has a Lipschitz extension
to the whole of $X$ with a minimal loss in the Lipschitz constant
$K$. Some of the first solutions of the Lipschitz extension problem
were complementary to the Tietze extension theorem and dated back to
the works of Hahn and Banach in the 1920s, and McShane, Whitney and
Kirszbraun in the 1930s. The classical Hahn-Banach theorem is based on
a successive point-by-point procedure of extending bounded linear
functionals. In the setting of a general metric domain, the conditions
are less restrictive and the extension is only required to be
Lipschitz with the same Lipschitz constant. In this case, the
successive procedure can be replaced by a much simpler one. The
following theorem was obtained by McShane in \cite[Theorem
1]{MR1562984}, see also Whitney \cite[the footnote on
p.\,63]{MR1501735}, and is commonly called the \emph{McShane-Whitney
  extension theorem}.

\begin{theorem}
  \label{theorem-Lipschitz-ext:12}
  Let $(X,d)$ be a metric space and $A\subseteq X$. Then each
  $K$-Lipschitz function $\varphi:A\to \R$ can be extended to a
  $K$-Lipschitz function $f:X\to \R$.
\end{theorem}

The extension in \Cref{theorem-Lipschitz-ext:12} was defined by an
explicit formula involving the function $\varphi$, the Lipschitz
constant $K$ and the metric $d$. For an extended discussion on this
theorem, the interested reader is referred to
\cite{Gutev2020,Gutev2025a}.\medskip

A map $f:X\to Y$ between metric spaces $(X,d)$ and $(Y,\rho)$ is
called \emph{nonexpansive} if $\lip(f)\leq 1$, i.e.\ when
$\rho(f(p),f(q))\leq d(p,q)$ for every $p,q\in X$.  If $(E,\|\cdot\|)$
is a linear normed space and $f:X\to E$ is a $K$-Lipschitz map for
some $K>0$, then the associated map $\frac1{K} f:X\to E$ is
nonexpansive. Thus, for a linear normed range $E$, the Lipschitz
extension problem is reduced to that of extending nonexpansive
maps. In the special case of the Euclidean space $E=\R^n$, the
McShane-Whithney extension theorem (\Cref{theorem-Lipschitz-ext:12})
implies the following consequence.

\begin{corollary}
  \label{corollary-Lipschitz-ext:2}
  Let $(X,d)$ be a metric space, $A\subseteq X$ and
  $\varphi:A\to \R^n$ be a nonexpansive map. Then $\varphi$ can be
  extended to a Lipschitz map $f:X\to \R^n$ such that\/
  $\lip(f)\leq \sqrt{n}$.
\end{corollary}

\begin{proof}
  Since each coordinate function $\varphi_i:A\to \R$, $i\leq n$, is
  $1$-Lipschitz being nonexpansive, by
  \Cref{theorem-Lipschitz-ext:12}, it can be extended to a
  $1$-Lipschitz function $f_i:X\to \R$. Then the resulting map
  $f=(f_1,\dots,f_n):X\to \R^n$ is as required.
\end{proof}

There are simple examples showing that for a general metric domain
$(X,d)$, the Lipschitz constant of the extension $f:X\to \R^n$ in
\Cref{corollary-Lipschitz-ext:2} is not necessarily equal to that of
$\varphi:A\to \R^n$. For instance, the following example was given in
\cite[2.10.44]{zbMATH03280855} and \cite[Example
4.2.1]{zbMATH07045619}.

\begin{example}
  \label{example-Lipschitz-ext:2}
  Let
  $\mathbf{u}_1=(1,-1), \mathbf{u}_2=(1,1),
  \mathbf{u}_3=\left(-1,1\right)\in \R^2$ and
  $\mathbf{0}=\left(0,0\right)\in\R^2$ be the origin. If
  $\|\mathbf{u}\|_\infty=\max\left\{|u_1|,|u_2|\right\}$ is the
  $\infty$-norm on $\R^2$, then
  \begin{equation}
    \label{eq:Lipschitz-ext:13}
    \|\mathbf{u}_1-\mathbf{u}_2\|_\infty=
    \|\mathbf{u}_2-\mathbf{u}_3\|_\infty =
    \|\mathbf{u}_3-\mathbf{u}_1\|_\infty= 2\ \ \text{and}\
    \ \|\mathbf{u}_i-\mathbf{0}\|_\infty=1,\ i\leq3.
  \end{equation}
  Hence, the set
  $A=\left\{\mathbf{u}_1,\mathbf{u}_2,\mathbf{u}_3\right\}$ represents
  the vertices of an ``equilateral triangle'' in $\R^2$ with respect
  to the $\infty$-norm. Take any equilateral triangle in the Euclidean
  plane $\R^2$ with a side of $2$, for instance $\mathbf{v}_1=(1,0)$,
  $\mathbf{v}_2=\left(0,\sqrt{3}\right)$ and
  $\mathbf{v}_3=(-1,0)$. Next, let
  $\mathbf{b}=\frac{\mathbf{v}_1+\mathbf{v}_2+\mathbf{v}_3}3=
  \left(0,\frac{\sqrt{3}}3\right)$ be its barycentre. Then all
  distances $\|\mathbf{b}-\mathbf{v}_i\|=\frac2{\sqrt{3}}$, $i\leq 3$,
  are equal.  Since $\mathbf{b}$ is also the Fermat point of this
  triangle\,---\,the point minimising the sum of distances to the
  triangle's vertices\,---\,it follows that
  \begin{equation}
    \label{eq:Lipschitz-ext:4}
    \max_{1\leq i\leq 3}\|\mathbf{y}-\mathbf{v}_i\|\geq
    \mfrac2{\sqrt{3}}\quad \text{for 
      every\, $\mathbf{y}\in \R^2$.}
  \end{equation}
  Next, define a map $\varphi:A\to \R^2$ by
  $\varphi(\mathbf{u}_i)=\mathbf{v}_i$, $i\leq 3$. Then $\varphi$ is
  an isometry as a map from the subset $A$ of
  $\left(\R^2,\|\cdot\|_\infty\right)$ to
  $\left(\R^2,\|\cdot\|\right)$. Hence, in particular, it is also
  nonexpansive. Let $f:\R^2\to \R^2$ be any extension of $g$. Then by
  \cref{eq:Lipschitz-ext:13,eq:Lipschitz-ext:4},
  \[
    \max_{1\leq i\leq 3}\|f(\mathbf{0})-f(\mathbf{u}_i)\|= \max_{1\leq
      i\leq 3}\|f(\mathbf{0})-\mathbf{v}_i\|\geq
    \mfrac2{\sqrt{3}}>1=\max_{1\leq i\leq
      3}\|\mathbf{0}-\mathbf{u}_i\|_\infty.
  \]
  Thus, $f$ cannot be nonexpansive.\qed
\end{example}

In contrast to \Cref{example-Lipschitz-ext:2}, the McShane-Whithney
extension theorem remains valid for maps between Euclidean spaces. The
following extension result was obtained by Kirszbraun
\cite{zbMATH02538492} and is commonly called the \emph{Kirszbraun
  extension theorem}.

\begin{theorem}
  \label{theorem-Lipschitz-ext:1}
  If $A\subseteq \R^m$, then each nonexpansive map
  $\varphi\colon A\to \R^n$ can be extended to a nonexpansive map
  $f:\R^m\to \R^n$.
\end{theorem}

\Cref{theorem-Lipschitz-ext:1} was rediscovered in a series of papers
by Valentine \cite{zbMATH03102004,zbMATH03102005,zbMATH03102006},
where he also generalised it to Hilbert and other classes of
spaces. Alternative proofs and further generalisations have been given
by a number of other authors, for instance by Mickle
\cite{zbMATH03049820}, Minty
\cite{zbMATH03180487,zbMATH03172151,zbMATH03304171}, Gr\"unbaum
\cite{zbMATH03210340}, Gr\"unbaum and Zarantonello
\cite{Gruenbaum1968}, Akopyan and Tarasov
\cite{zbMATH05941757}.\medskip

The proof of \Cref{theorem-Lipschitz-ext:1} is based on one-point
extensions, which are discussed in the next section. Briefly, in the
next section, we present two interpretations of the Kirszbraun
extension theorem, formulated solely in terms of the intersection of a
finite number of closed balls in Euclidean spaces\,---\,a property
known as the \emph{Kirszbraun intersection property} (see
\Cref{theorem-Lipschitz-ext:2,theorem-Lip-Ball-Intersc:1}). In
\Cref{sec:best-one-point}, we consider best one-point extensions in
Euclidean spaces from finite subsets of the domain, and show that the
Kirszbraun intersection property is achieved by points that are
\emph{closest} to the range (see \Cref{theorem-KIP-ver-15:1}).  The
section also contains an overview of all important concepts related to
Euclidean spaces that are relevant to this result, and leads to
\Cref{lemma-Lipschitz-ext:1}, which lays the groundwork for such
extensions. \Cref{sec:kirszbr-valent-exten} contains the proof of the
\emph{Kirszbraun-Valentine extension theorem}
(\Cref{theorem-Lipschitz-ext:6}) in Hilbert spaces as a consequence of
the Kirszbraun intersection property in Euclidean spaces. This
reduction is based on the finite intersection property of closed
bounded convex sets in Hilbert spaces
(\Cref{theorem-Lipschitz-ext:7}). \Cref{sec:scho-solut-meng} presents
a brief overview of Menger's embedding problem of finite metric spaces
into Euclidean spaces, and the Schoenberg solution of this problem
(\Cref{theorem-Lip-Ext-Embb:1}). Schoenberg's approach to the Menger
embedding problem underlies a relatively simple metric
condition\,---\,proposed by Gr\"unbaum and Zarantonello in
\cite{Gruenbaum1968}\,---\,for generalising the Kirszbraun-Valentine
extension theorem from a Hilbert domain to a metric one (see
\Cref{theorem-KIP-ver-17:1}). In particular, this condition gives a
very simple proof of the \emph{Mickle extension theorem} (see
\Cref{corollary-KIP-ver-24:1}). In \Cref{sec:foxs-ineq-hold}, we
consider the extension problem of H\"older continuous maps in Hilbert
spaces. The solution is based on an inequality credited to Fox, see
\cref{eq:Kirszbraun-Intersc-v3:1}, and the aforementioned
Gr\"unbaum-Zarantonello condition. Thus, we also obtain a very simple
proof of the \emph{Minty extension theorem} (see
\Cref{theorem-Kirszbraun-Intersc-v4:2}). \Cref{sec:conc-funct-affine}
is devoted to basic properties of concave functions. The main results
of this section are that each concave function can be characterised by
its affine majorants (\Cref{theorem-Rem-Lipschitz:1}) and that, under
certain conditions, any increasing function possesses a concave
majorant (\Cref{theorem-Uniform-Cont-Ext:1}). Finally, in
\Cref{sec:grunb-zarant-extens}, we give a very simple and
self-contained proof of the \emph{Gr\"unbaum-Zarantonello extension
  theorem} for uniformly continuous maps in Hilbert spaces (see
\Cref{theorem-Uni-Ball-Intersc:1}).

\section{The Kirszbraun Intersection Property}

The proof of \Cref{theorem-Lipschitz-ext:1} is based on the successive
procedure of point-by-point extensions, so the extension over the
entire space is essentially the same as that in proof of the
Hahn-Banach theorem. However, in the setting of
\Cref{theorem-Lipschitz-ext:1}, one-point extensions are not so
straightforward as those in the Hahn-Banach theorem. At the heart of
these one-point extensions lies a geometric property related to the
intersection of closed balls, which constitutes the most difficult
part of the proof. The property itself is based on the following
considerations.\medskip

Each Lipschitz map is uniformly continuous. Extensions of uniformly
continuous maps came naturally in the context of the completion
problem of metric spaces. The following result appeared in Felix
Hausdorff's 1914 monograph "Grundz\"uge der Mengenlehre"
\cite{hausdorff:14}.

\begin{theorem}
  \label{theorem-Lipschitz-ext:3}
  If\/ $\varphi$ is a uniformly continuous map into a complete metric
  space defined on a subspace $A$ of a metric space, then there exists
  a continuous extension of\/ $\varphi$ defined on the closure of $A$.
\end{theorem}

However, the fact that the extension to the \emph{closure}
$\overline{A}$ of $A$ is also uniformly continuous was not stated
explicitly, and it took more than a decade to realise this. The
development of uniformly continuous extensions was gradual, the
interested reader is referred to \cite{Bentley_1998} which contains an
interesting overview on the history of this extension problem.\medskip

One can easily see that a continuous map between metric spaces is
uniformly continuous iff it is uniformly continuous on some (every)
dense subset of its domain. Similarly, a continuous map between metric
spaces is $K$-Lipschitz iff it is $K$-Lipschitz on some (every) dense
subset of its domain. Accordingly, in the setting of Lipschitz maps,
we have the following special case of \Cref{theorem-Lipschitz-ext:3}.

\begin{corollary}
  \label{corollary-Lipschitz-ext:3}
  Let $\varphi:A\to Y$ be a $K$-Lipschitz map from a subset
  $A\subseteq X$ of a metric space $(X,d)$ to a complete metric space
  $(Y,\rho)$. Then $\varphi$ has a $K$-Lipschitz extension
  $\psi:\overline{A}\to Y$ to the closure of $A$.
\end{corollary}

Thus, for a complete metric range, the Lipschitz extension problem is
reduced only to Lipschitz maps defined on closed subsets.  Hence, in
this setting, it is a special continuous extension problem where the
continuity of the extension is controlled by the Lipschitz
constant. On the other hand, the Lipschitz extension problem is not a
topological problem. In fact, it is not only a metric problem, but
also the Lipschitz condition does not depend on the triangle
inequality of the given metrics. Ignoring this property of metrics,
Lipschitz maps can be defined in the following more general setting. A
function $d:X\times X\to [0,+\infty)$ on a set $X$ is called a
\emph{symmetric} on $X$ (in some sources also a \emph{semi-metric} on
$X$) if it is a symmetric function and $d(x,y)=0$ if and only if
$x=y$. In other words, $d$ satisfies all axioms of a metric except,
possibly, the triangular inequality. In what follows, each pair
$(X,d)$ consisting of a set $X$ and a symmetric function $d$ on it
will be referred to simply as a \emph{symmetric space}. Unlike
\emph{symmetrizable spaces} in topology, we will not be interested in
any topological structure on a symmetric space $(X,d)$ that is
associated with the symmetric $d$ on $X$.\medskip

Lipschitz maps between symmetric spaces are defined in the same way as
those between metric spaces. Namely, a map $f:X\to Y$ between
symmetric spaces $(X,d)$ and $(Y,\rho)$ is \emph{$K$-Lipschitz} for
some $K\geq 0$ if $\rho(f(p),f(q))\leq K d(p,q)$ for every $p,q\in
X$. In the special case when $K=1$, as before, we will say that such a
map $f:X\to Y$ is \emph{nonexpansive}. Let us explicitly remark that
this is not a formal generalisation per se, but will play an important
role in the extension problem of nonexpansive maps.

\begin{theorem}
  \label{theorem-Lipschitz-ext:9}
  Let $(X,d)$ and $(Y,\rho)$ be symmetric spaces and $K\geq 0$ be such
  that for every $S\subseteq X$ and $p\in X\setminus S$, every
  $K$-Lipschitz map $\varphi : S\to Y$ can be extended to a
  $K$-Lipschitz map\/ $\psi : S\cup\{p\}\to Y$. Then for every
  $A\subseteq X$, each $K$-Lipschitz map $\varphi:A\to Y$ can be
  extended to a $K$-Lipschitz map\, $\Phi:X\to Y$.
\end{theorem}

\Cref{theorem-Lipschitz-ext:9} has a standard proof by using a
suitable interpretation of the Axiom of Choice. For instance, for
$A\subseteq X$ and a $K$-Lipschitz map $\varphi:A\to Y$, consider the
collection $\Omega$ of all partial $K$-Lipschitz extensions of
$\varphi$, i.e.\ $\psi\in \Omega$ precisely when $A$ is a subset of
the domain $\dom(\psi)$ of $\psi$ and $\psi\uhr A=\varphi$. Then
$\Omega$ is partially ordered with respect to the relation
$\psi_1\leq \psi_2$ defined by $\dom(\psi_1)\subseteq \dom(\psi_2)$
and ${\psi_2\uhr \dom(\psi_1)=\psi_1}$. In terms of this partial
order, each nonempty chain $\mathcal{H}\subseteq \Omega$ has an upper
bound in $\Omega$. Indeed, set
$Z=\bigcup_{\psi\in \mathcal{H}}\dom(\psi)$ and define ${\eta:Z\to Y}$
by $\eta\uhr \dom(\psi)=\psi$, for $\psi\in \mathcal{H}$. Then $\eta$
is also $K$-Lipschitz because for each pair of points $p,q\in Z$ there
exists $\psi\in \mathcal{H}$ with $p,q\in \dom(\psi)$, and therefore
\[
  \rho(\eta(p),\eta(q))=\rho(\psi(p),\psi(q))\leq K d(p,q).
\]
Thus, $\eta\in \Omega$ and clearly $\psi\leq\eta$ for every
$\psi\in \mathcal{H}$. Hence, according to Zorn's lemma, the map
$\varphi$ has a maximal $K$-Lipschitz extension $\Phi\in \Omega$. We
must now have that $\dom(\Phi)=X$. Otherwise, if $\dom(\Phi)\neq X$,
we may apply the given property of one-point extensions to get a
contradiction with the maximality of $\Phi$.\medskip

Regarding the role of the Axiom of Choice in
\Cref{theorem-Lipschitz-ext:9}, let us point out the following special
case in which the use of this axiom can be avoided.

\begin{remark}
  \label{remark-Lip-Ball-Intersc:1}
  For a separable metric space $(X,d)$ and a complete metric space
  $(Y,\rho)$, there is a direct proof of
  \Cref{theorem-Lipschitz-ext:9} based on a usual induction. Namely,
  take a dense subset $\{p_n:n\in \N\}\subseteq X\setminus A$, and
  next set $S_0=A$ and $S_n=S_{n-1}\cup \{p_n\}$, $n\in\N$. If
  $\varphi_0=\varphi:A\to Y$ is a $K$-Lipschitz map, then by induction
  using the hypothesis in \Cref{theorem-Lipschitz-ext:9}, for each
  $n\in\N$ there exists a $K$-Lipschitz map $\varphi_n:S_n\to Y$ such
  that $\varphi_n\uhr S_{n-1}=\varphi_{n-1}$. Evidently,
  $\psi\uhr S_n=\varphi_n$, $n\in \N$, defines a $K$-Lipschitz map
  $\psi:S\to Y$ over the set $S=\bigcup_{n=0}^\infty S_n$. Since $S$
  is dense in $X$, according to \Cref{corollary-Lipschitz-ext:3},
  $\psi$ can be extended to a $K$-Lipschitz map $\Phi:X\to Y$. This
  $\Phi$ is as required.\qed
\end{remark}

For a symmetric space $(Y,\rho)$ and $\delta\geq 0$, we will use
$\mathbf{C}(p,\delta)=\left\{y\in Y: d(y,p)\leq \delta\right\}$ to
denote the \emph{closed $\delta$-ball centred at a point $p\in
  Y$}. Since it is not assumed a priori that there is a topological
structure on $Y$, the term ``closed'' reflects only the defining
property of this $\delta$-ball with respect to the linear order
``$\leq$''on the real line $\R$. In terms of these balls, the
following simple observation is complementary to
\Cref{theorem-Lipschitz-ext:9}.

\begin{proposition}
  \label{proposition-Lip-Ball-Intersc:4}
  Let $(X,d)$ and $(Y,\rho)$ be symmetric spaces, $S\subseteq X$ be a
  nonempty set, $\varphi:S\to Y$ be a nonexpansive map and
  $p\in X\setminus S$. Then an extension $\psi:S\cup\{p\}\to Y$ of
  $\varphi$ is nonexpansive if and only if\/
  $\psi(p)\in \bigcap_{x\in S} \mathbf{C}\left(\varphi(x),
    d(p,x)\right)$. In particular, $\varphi$ can be extended to a
  nonexpansive map over $S\cup \{p\}$ precisely when
  $\bigcap_{x\in S}\mathbf{C}\big(\varphi(x),d(p,x)\big)\neq
  \emptyset$.
\end{proposition}

A family $\mathcal{F}$ of subsets of a set $Y$ has the \emph{finite
  intersection property} if ${\bigcap \mathcal{S}\neq \emptyset}$ for
every nonempty finite subset $\mathcal{S}\subseteq \mathcal{F}$. In
the Euclidean space $\R^n$, each closed ball
$\mathbf{C}(\mathbf{v},\delta)=\left\{\mathbf{u}\in \R^n:
  \|\mathbf{u}-\mathbf{v}\|\leq \delta\right\}$, for
$\mathbf{v}\in \R^n$ and $\delta\geq 0$, is compact. Hence, taking
$Y=\R^n$ in \Cref{proposition-Lip-Ball-Intersc:4}, the family
$\left\{\mathbf{C}\big(\varphi({x}), d(p,x)\big): x\in S\right\}$
consists of compact sets and, consequently, it has a nonempty
intersection precisely when it has the finite intersection
property. Thus,
$\bigcap_{x\in S}\mathbf{C}\big(\varphi(x),d(p,x)\big)\neq \emptyset$
precisely when for each nonempty finite set $T\subseteq S$, each
nonexpansive map $\varphi:T\to \R^n$ can be extended to a nonexpansive
map $\psi:T\cup\{{p}\}\to \R^n$. Therefore, according to
\Cref{theorem-Lipschitz-ext:9} and
\Cref{proposition-Lip-Ball-Intersc:4}, this implies the following
general reduction of the extension problem of nonexpansive maps in
Euclidean spaces.

\begin{corollary}
  \label{corollary-Kirszbraun-IP-ver-01:1}
  Let $(X,d)$ be a symmetric space such that for every finite subset
  $T\subseteq X$ and $p\in X\setminus T$, each nonexpansive map
  $\varphi:T\to \R^n$ can be extended to a nonexpansive map
  $\psi:T\cup\{{p}\}\to \R^n$. Then for every $A\subseteq X$, each
  nonexpansive map $\varphi:A\to \R^n$ can be extended to a
  nonexpansive map $f:X\to \R^n$.
\end{corollary}

It is also obvious that when $(X,d)$ is a separable metric space, we
can apply \Cref{remark-Lip-Ball-Intersc:1} instead of the full version
of \Cref{theorem-Lipschitz-ext:9}. Therefore, such simple arguments
show that \Cref{theorem-Lipschitz-ext:1} is equivalent to the
following special case of one-point extensions.

\begin{theorem}
  \label{theorem-Lipschitz-ext:2}
  Let $T\subseteq \R^m$ be a finite set. Then for each point\/
  $\mathbf{x}\in \R^m\setminus T$, each nonexpansive map
  $\varphi:T\to \R^n$ can be extended to a nonexpansive map
  $\psi:T\cup\{\mathbf{x}\}\to \R^n$.
\end{theorem}

Furthermore, suppose that
$T=\left\{\mathbf{u}_1,\dots, \mathbf{u}_k\right\}\subseteq \R^m$ and
$\mathbf{v}_i=\varphi(\mathbf{u}_i)$, $i\leq k$, for some map
$\varphi:T\to \R^n$. Then it follows trivially that the map
$\varphi:T\to \R^n$ is nonexpansive precisely when
$\|\mathbf{v}_i-\mathbf{v}_j\|\leq \|\mathbf{u}_i-\mathbf{u}_j\|$ for
each $i,j\leq k$. Accordingly, \Cref{theorem-Lipschitz-ext:2} can be
further reformulated without explicitly mentioning any nonexpansive
maps.

\begin{theorem}
  \label{theorem-Lip-Ball-Intersc:1}
  Let\, $\mathbf{u}_i\in \R^m$ and\, $\mathbf{v}_i\in \R^n$,
  $1\leq i\leq k$, be vectors such that
  \[
    \|\mathbf{v}_i-\mathbf{v}_j\|\leq \|\mathbf{u}_i-\mathbf{u}_j\|\quad
    \text{for every $i,j\leq k$.} 
  \]
  Then for every\, $\mathbf{x}\in \R^m$ there exists\,
  $\mathbf{y}\in \R^n$ such that
  $\|\mathbf{y}-\mathbf{v}_i\|\leq \|\mathbf{x}-\mathbf{u}_i\|$ for
  every $i\leq k$.
\end{theorem}

The property in either of the above two theorems is often called the
\emph{Kirszbraun intersection property}, it was established by
Kirszbraun in \cite{zbMATH02538492}.\medskip

A simple proof of \Cref{theorem-Lip-Ball-Intersc:1} was given by
Schoenberg \cite[Theorem 2]{zbMATH03084778}, and further simplified in
\cite[Lemma 2.10.40 and Theorem 2.10.43]{zbMATH03280855}. An important
role in all these proofs is played by properties of convex sets in the
Euclidean space $\R^n$. This proof and these properties of convex sets
are discussed in detail in the next section.

\section{Best One-Point Extensions in Euclidean Spaces}
\label{sec:best-one-point}

Here, we will slightly generalise \Cref{theorem-Lipschitz-ext:2} by
showing that not only does some nonexpansive extension exist, but also
that there exists a ``best'' such extension, see
\Cref{theorem-KIP-ver-15:1}. This will be based on the following two
interpretations of closest points in metric spaces.\medskip

Let $(X,d)$ be a metric space. The distance from a point $p\in X$ to a
nonempty subset $A\subseteq X$ is defined by
$d(p,A)=\inf\left\{d(p,x): x\in A\right\}$. This distance not only
determines a nonexpansive function $d(\cdot,A):X\to \R$, but is also a
tool for defining special points of $A$ and, in particular, special
subsets of $X$. Namely, a point $a\in A$ is called \emph{closest} (or
\emph{nearest}) to $p\in X$ if $d(p,a)=d(p,A)$, and $A\subseteq X$ is
called \emph{proximinal} if each $p\in X$ admits a closest point of
$A$. Evidently, $A\subseteq X$ is proximal precisely when
\begin{equation}
  \label{eq:KIP-ver-03:1}
  d(p,A)=\min\left\{d(p,x): x\in A\right\}\quad \text{for every $p\in
    X$.} 
\end{equation}
The term ``proximal'' was proposed by Killgrove and first used by
Phelps in \cite{MR87897}. Each proximal set $A\subseteq X$ is closed
because $A$ must contain all points at zero distance from $A$. The
converse is not necessarily true, and one can easily construct
nonempty closed sets in metric spaces that are not proximal. In
contrast, every nonempty compact subset of $X$ is proximal, which
follows easily from \cref{eq:KIP-ver-03:1} and the fact that the
metric $d:X\times X\to \R$ is continuous. By the Heine-Borel property,
every closed ball of $\R^n$ is compact. Therefore, every nonempty
closed subset of $\R^n$ is proximal.\medskip

For a metric space $(X,d)$ and a nonempty subset $A\subseteq X$,
another concept of ``closest'' points, now to subsets, was introduced
by Phelps in \cite{MR104139}.  Namely, a point $q\in X$ is called
\emph{pointwise closer} to $A$ than $p\in X$ \cite{MR104139} if
$d(q,x)<d(p,x)$ for every $x\in A$.  If a point $p\in X$ is such that
no other point of $X$ is pointwise closer to $A$ than is $p$, then it
is called a \emph{closest} point to $A$ \cite{MR104139}. These points
are naturally related to the intersection of open and closed balls in
metric spaces. For this purpose, for $\delta>0$, we will use
$\mathbf{O}(p,\delta)=\left\{x\in X: d(x,p)< \delta\right\}$ to denote
the \emph{open $\delta$-ball centred at a point $p\in X$}. It is now
obvious that a point $q\in X$ is pointwise closer to $A$ than a point
$p\in X$ precisely when
$q\in \bigcap_{x\in A}\mathbf{O}\left(x,d(p,x)\right)$. Accordingly,
$p\in X$ is a closest point to the subset $A$ if and only if
$\bigcap_{x\in A}\mathbf{O}\left(x,d(p,x)\right)=\emptyset$.  However,
we always have that
$p\in \bigcap_{x\in A}\mathbf{C}\left(x,d(p,x)\right)$. Hence, in view
\Cref{proposition-Lip-Ball-Intersc:4}, this is naturally related to
one-point extensions of nonexpansive maps. Indeed, in this section we
will prove the following theorem.

\begin{theorem}
  \label{theorem-KIP-ver-15:1}
  Let $T\subseteq \R^m$ be a nonempty finite set,
  $\mathbf{x}\in \R^m\setminus T$ and $\varphi:T\to \R^n$ be a
  nonexpansive map. Then $\varphi$ can be extended to a nonexpansive
  map $\psi:T\cup\{\mathbf{x}\}\to \R^n$ such that\,
  $\psi(\mathbf{x})\in \R^n$ is a closest point to the subset\/
  $\varphi(T)\subseteq \R^n$.
\end{theorem}

The extension in \Cref{theorem-KIP-ver-15:1} will be constructed in an
explicit way based on proximal subsets in Euclidean spaces, which may
also explain its uniqueness-like property. To prepare for this proof,
we will first give a brief account on some basic properties of the
Euclidean norm.\medskip

In contrast to general metric spaces, for a linear normed space
$(E,\|\cdot\|)$, a point $v\in E$ and $\delta>0$, we always have that
$\mathbf{C}(v,\delta)=\overline{\mathbf{O}(v,\delta)}$. Simply, if
$\|{u}-{v}\|=\delta$ for some $u\in E$, then
${(1-\lambda){u}+\lambda {v}\in \mathbf{O}({v},\delta)}$ for every
$0<\lambda\leq 1$. In fact, based on this idea, we have the following
related property about the closure of the intersection of open convex
subsets in linear normed spaces.

\begin{proposition}
  \label{proposition-Lip-Ball-Intersc:3}
  Let $\mathcal{U}$ be a family of open convex subsets of a linear
  normed space $(E,\|\cdot\|)$ such that\,
  $\bigcap \mathcal{U}\neq \emptyset$. Then\,
  $\overline{\bigcap \mathcal{U}}=\bigcap_{U\in
    \mathcal{U}}\overline{U}$.
\end{proposition}

\begin{proof}
  If ${v}\in \bigcap \mathcal{U}$ and
  ${u}\in \bigcap_{U\in \mathcal{U}}\overline{U}$, then
  ${u}_n=\left(1-\frac1n\right) {u}+\frac1n{v}\in \bigcap\mathcal{U}$
  for every $n\in \N$. Accordingly,
  ${u}=\lim_{n\to\infty} {u}_n\in \overline{\bigcap \mathcal{U}}$.
\end{proof}

In the special case of the Euclidean space $E=\R^n$, another important
role will be played by the fact that the Euclidean norm is
\emph{strictly convex}, i.e.\ that the \emph{unit sphere}\,
$\s^{n-1}=\left\{\mathbf{x}\in \R^{n}:
  \|\mathbf{x}\|=1\right\}\subseteq \R^n$ does not contain any line
segment. The reason for this property lies in the fact that this norm
measures not only the distance between vectors in $\R^n$, but also the
angle between them. Indeed, the Euclidean norm is simply given by
$\|\mathbf{u}\|=\sqrt{\langle\mathbf{u},\mathbf{u}\rangle}$, where
$\langle\cdot,\cdot\rangle$ is the \emph{inner product} on $\R^n$
defined by
\[
  \langle\mathbf{u},\mathbf{v}\rangle=\sum_{i=1}^n u_iv_i,\quad
  {\mathbf{u}=(u_1,\dots, u_n),\mathbf{v}=(v_1,\dots, v_n)\in \R^n}.
\]
Accordingly, for nonzero vectors $\mathbf{u},\mathbf{v}\in \R^n$, the
Cauchy-Schwarz inequality implies that
${\frac{|\langle\mathbf{u},\mathbf{v}\rangle|}{\|\mathbf{u}\|\cdot
    \|\mathbf{v}\|}\leq 1}$. Therefore, there is a unique angle
$\alpha_{\mathbf{uv}}\in [0,\pi]$ with
$\cos(\alpha_{\mathbf{uv}})=\frac{\langle\mathbf{u},\mathbf{v}\rangle}
{\|\mathbf{u}\|\cdot \|\mathbf{v}\|}$, i.e.\ for which
$\langle\mathbf{u},\mathbf{v}\rangle= \|\mathbf{u}\|\cdot
\|\mathbf{v}\|\cos(\alpha_{\mathbf{uv}})$. The angle
$\alpha_{\mathbf{uv}}$ is usually called the \emph{unoriented} (or
\emph{undirected}) angle between these vectors and serves to reproduce
the well-known \emph{law of cosines} from elementary
trigonometry. This is now expressed in the following more general form
that
\[
  \|\mathbf{u}\pm\mathbf{v}\|^2= \|\mathbf{u}\|^2 + \|\mathbf{v}\|^2
  \pm 2\langle\mathbf{u},\mathbf{v}\rangle\quad \text{for every
    $\mathbf{u},\mathbf{v}\in \R^n$.}
\]
An alternative form of this equality is also familiar from elementary
geometry, it represents the well-known \emph{parallelogram law} that
\[
  \|\mathbf{u}+\mathbf{v}\|^2
  +\|\mathbf{u}-\mathbf{v}\|^2=2\|\mathbf{u}\|^2+2\|\mathbf{v}\|^2\quad
  \text{for every $\mathbf{u},\mathbf{v}\in \R^n$.}
\]

Strictly convex sets in Euclidean spaces can now be reformulated into
the following equivalent form only in terms of open and closed balls.

\begin{proposition}
  \label{proposition-Lip-Ball-Intersc:5}
  If\, $\mathbf{z}\in\R^n$ and $\delta>0$, then
  \begin{equation}
    \label{eq:Lip-Ball-Intersc:4}
    \mfrac{\mathbf{x}+\mathbf{y}}2\in
    \mathbf{O}(\mathbf{z},\delta)\quad \text{for 
      every\, $\mathbf{x},\mathbf{y}\in \mathbf{C}(\mathbf{z},\delta)$\/
      with\, $\mathbf{x}\neq 
      \mathbf{y}$.} 
  \end{equation}
\end{proposition}

\begin{proof}
  Apply the parallelogram law with
  $\mathbf{u}=\frac{\mathbf{x}-\mathbf{z}}2$ and
  $\mathbf{v}=\frac{\mathbf{y}-\mathbf{z}}2$.
\end{proof}

For a metric space $(X,d)$, a proximal subset $A\subseteq X$ is called
\emph{uniquely proximal} if each $p\in X$ admits a unique closest
point of $A$. Nowadays, uniquely proximal sets are usually called
\emph{Chebyshev}, named after Pafnuty Chebyshev, see Efimov and
Ste\v{c}kin \cite{MR0095445}. In the setting of the Euclidean space
$\R^n$, each convex proximal subset $A\subseteq \R^n$ is Chebyshev,
which follows easily from the property in
\cref{eq:Lip-Ball-Intersc:4}. In fact, the Chebyshev subsets of $\R^n$
are precisely the nonempty closed convex subsets of $\R^n$. This is
known as the \emph{Bunt-Kritikos-Motzkin theorem}. It was obtained by
Bunt \cite{zbMATH03011214} in 1934 and, independently, by Kritikos
\cite{zbMATH02519033} in 1938. The special case of $n=2$ was obtained
by Motzkin \cite{zbMATH03017862} in 1935, while an alternative proof
was also given by Jessen \cite{MR0003721} in 1940.\medskip

The following characteristic property of closest points will play an
important role in the proof of \Cref{theorem-KIP-ver-15:1}. The
inequality in this property is known as the \emph{Kolmogorov
  condition}. Its prototype goes back to Kolmogorov's work in
\cite{MR0025609}. 

\begin{theorem}
  \label{theorem-Lipschitz-ext:4}
  Let $A\subseteq \R^n$ be a convex set and\/ $\mathbf{y}\in
  \R^n$. Then a point\/ $\mathbf{a}\in A$ is closest to\/ $\mathbf{y}$
  if and only if
  \begin{equation}
    \label{eq:Lipschitz-ext:10}
    \left\langle\mathbf{y}-\mathbf{a},\mathbf{x}-\mathbf{a}\right\rangle
    \leq 
    0\quad \text{for every\, $\mathbf{x}\in A$.}
  \end{equation}
\end{theorem}

\begin{proof}
  Suppose that $\mathbf{a}\in A$ is a closest point to $\mathbf{y}$,
  and take $\mathbf{x}\in A$ and $0< \lambda< 1$. Then, since $A$ is
  convex, $(1-\lambda)\mathbf{a}+ \lambda \mathbf{x}\in A$ and it
  follows that
  \begin{align*}
    \|\mathbf{y}-\mathbf{a}\|^2
    &\leq
      \|\mathbf{y}-[(1-\lambda)\mathbf{a}
      +\lambda\mathbf{x})]\|^2\\
    &= \|(\mathbf{y}-\mathbf{a})-
      \lambda(\mathbf{x}-\mathbf{a})\|^2\\
    &=\|\mathbf{y}-\mathbf{a}\|^2
      -2\lambda \langle\mathbf{y}-
      \mathbf{a},\mathbf{x}-\mathbf{a}\rangle 
      +\lambda^2\left\|\mathbf{x}- \mathbf{a}\right\|^2.
  \end{align*}
  Hence,
  $2\langle\mathbf{y}-\mathbf{a},\mathbf{x}-
    \mathbf{a}\rangle \leq\lambda\|\mathbf{x}- \mathbf{a}\|^2$
  for each $0<\lambda<1$, which is the inequality in
  \cref{eq:Lipschitz-ext:10}.\medskip

  Conversely, suppose that $\mathbf{a}\in A$ is as in
  \cref{eq:Lipschitz-ext:10}, and take any point $\mathbf{x}\in
  A$. Then
  \begin{align*}
    \|\mathbf{y}-\mathbf{x}\|^2
    &= \|(\mathbf{y}-\mathbf{a})+
      (\mathbf{a}-\mathbf{x})\|^2 \\
    &=\|\mathbf{y}-\mathbf{a}\|^2
      +2\langle\mathbf{y}-\mathbf{a},
      \mathbf{a}-\mathbf{x}\rangle 
      +\|\mathbf{a}-\mathbf{x}\|^2\geq \|\mathbf{y}-\mathbf{a}\|^2,
  \end{align*}
  because
  $\langle\mathbf{y}-\mathbf{a},\mathbf{a}-\mathbf{x}\rangle \geq 0$.
\end{proof}

The inequality in \cref{eq:Lipschitz-ext:10} gives another proof that
each proximal convex subset $A\subseteq \R^n$ is Chebyshev. Indeed,
let $\mathbf{a}_1,\mathbf{a}_2\in A$ be closest points to some
$\mathbf{y}\in \R^n\setminus A$ for such a subset $A\subseteq
\R^n$. Then by \cref{eq:Lipschitz-ext:10},
$\langle\mathbf{y}-\mathbf{a}_1,\mathbf{a}_2-\mathbf{a}_1\rangle \leq
0$ and
$\langle\mathbf{y}-\mathbf{a}_2,\mathbf{a}_1-\mathbf{a}_2\rangle \leq
0$.  Since
$\langle\mathbf{a}_2-\mathbf{y},\mathbf{a}_2-\mathbf{a}_1\rangle =
\langle\mathbf{y}-\mathbf{a}_2,\mathbf{a}_1-\mathbf{a}_2\rangle$, we
get that
$\langle\mathbf{a}_2-\mathbf{a}_1,\mathbf{a}_2-\mathbf{a}_1\rangle
\leq 0$, i.e.\ that $\mathbf{a}_2=\mathbf{a}_1$.\smallskip

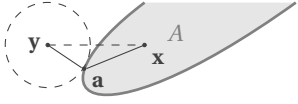
\begin{wrapfigure}[4]{l}{0.32\textwidth}
  \vspace{-5pt} \hspace{12pt}
  \begin{tikzpicture}[scale=0.8, every node/.style={scale=0.8}]
    \filldraw[black!10!white,line width=1pt] (3,1.5) .. controls (1,0)
    and (2,-1) .. (5.5,1.5);
    \draw[black!50!white,line width=1pt]
    (3,1.5) .. controls (1,0) and (2,-1) .. (5.5,1.5) (3.5,1)
    node[anchor=center] {$A$};
    \draw [darkgray, line width=.4pt,
    dashed] (1.4,.8) circle [radius=20pt];
    \filldraw [darkgray] (1.4,.8) circle [radius=.8pt]
    node[anchor=east] 
    {\small $\mathbf{y}$} (2,.41) circle [radius=.8pt]
    node[anchor=north west] {\small $\mathbf{a}$} (3,.8)
    circle [radius=.8pt]
    node[anchor=north west] {\small $\mathbf{x}$};
    \draw[darkgray,line width=.4pt] (1.4,.8) -- (1.96,.425);
    \draw[darkgray,line width=.4pt] (3,.8)-- (2.04,.41);
    \draw[darkgray,line width=.4pt, dashed] (1.4,.8)-- (3,.8);
  \end{tikzpicture}
\end{wrapfigure}
\noindent Furthermore, for $\mathbf{y}\neq \mathbf{a}\neq \mathbf{x}$,
the Kolmogorov condition in \cref{eq:Lipschitz-ext:10} represents the
property that the undirected angle between the vectors
$\mathbf{y}-\mathbf{a}$ and $\mathbf{x}-\mathbf{a}$ is obtuse. Hence,
$\mathbf{x}-\mathbf{y}=(\mathbf{x}-\mathbf{a})-(\mathbf{y}-\mathbf{a})$
is the longest side in the triangle formed by these vectors. Indeed,
we have the following consequence.

\begin{corollary}
  \label{corollary-Lipschitz-ext:1}
  Let $A\subseteq \R^n$ be a convex set and\/
  $\mathbf{y}\in \R^n\setminus A$. If a point\/ $\mathbf{a}\in A$ is
  closest to $\mathbf{y}$, then
  \begin{equation}
    \label{eq:Lipschitz-ext:11}
    \|\mathbf{x}-\mathbf{a}\|< \|\mathbf{x}-\mathbf{y}\|\quad
    \text{for every\, $\mathbf{x}\in A$.}
  \end{equation}
\end{corollary}

\begin{proof}
  If $\mathbf{a}\in A$ is a closest point to $\mathbf{y}$ and
  $\mathbf{x}\in A$, then \cref{eq:Lipschitz-ext:10} implies that
  \begin{align*}
    \|\mathbf{x}-\mathbf{y}\|^2
    &= \|(\mathbf{a}-\mathbf{y})
      +(\mathbf{x}-\mathbf{a})\|^2 \\
    &= \|\mathbf{a}-\mathbf{y}\|^2 +2\langle\mathbf{a}-\mathbf{y},
      \mathbf{x}-\mathbf{a}\rangle +\|\mathbf{x}-\mathbf{a}\|^2\\
    &\geq \|\mathbf{a}-\mathbf{y}\|^2
      +\|\mathbf{x}-\mathbf{a}\|^2>
      \|\mathbf{x}-\mathbf{a}\|^2.\qedhere
  \end{align*}
\end{proof}

The key construction in the proof of \Cref{theorem-KIP-ver-15:1} is
the following lemma, which is a simplified version of \cite[Lemma
2.10.40]{zbMATH03280855}, while its underlying idea can be traced back
to a work by Schoenberg\,---\,see \cite[Theorem 2]{zbMATH03084778}.
In this lemma, and what follows, $\conv(S)$ denotes the \emph{convex
  hull} of a subset $S \subseteq E$ of a vector space $E$.

\begin{lemma}
  \label{lemma-Lipschitz-ext:1}
  Let\/ $\mathbf{v}_u\in \R^n$ and $\delta_u>0$, $u\in T$, where $T$
  is a nonempty finite set. Then there exists\/ $\tau\geq 0$ and\,
  $\mathbf{y}\in \bigcap_{u\in
    T}\mathbf{C}(\mathbf{v}_u,\tau\delta_u)$ such that
  \begin{equation}
    \label{eq:Kirszbraun-Intersc-v1:1}
    \mathbf{y}\in\conv\left(\left\{\mathbf{v}_u: u\in T\ \text{and}\,\  
        \|\mathbf{y}-\mathbf{v}_u\|=\tau\delta_u\right\}\right). 
  \end{equation}
  Furthermore, $\mathbf{y}$ is a closest point to the set\/
  $\left\{\mathbf{v}_u:u\in T\right\}$.
\end{lemma}

\begin{proof}
  For each $t\geq 0$, let $Y_t\subseteq \R^n$ be the closed convex set
  defined by
  \begin{equation}
    \label{eq:Lipschitz-ext:14}
    Y_t=
    \bigcap_{u\in
      T}\mathbf{C}\left(\mathbf{v}_u,t\delta_u\right). 
  \end{equation}
  Since $Y_t\neq \emptyset$ for some $t\geq 0$, for instance
  $\mathbf{0}\in Y_t$ for every
  $t\geq \max_{u\in T}\frac{\|\mathbf{v}_u\|}{\delta_u}$, the number
  $\tau=\inf\left\{t\geq 0: Y_t\neq \emptyset\right\}$ is correctly
  defined. Moreover, by \cref{eq:Lipschitz-ext:14},
  $Y_\tau\neq \emptyset$ because ${Y_\tau=\bigcap_{t>\tau} Y_t}$ and
  $\left\{Y_t: t>\tau\right\}$ is a decreasing family of nonempty
  compact subsets of $\R^n$. Furthermore, if
  $\mathbf{w}\in \bigcap_{u\in
    T}\mathbf{O}\left(\mathbf{v}_u,t\delta_u\right)$ for some $t>0$,
  then
  $\tilde{t}= \max_{u\in
    T}\frac{\|\mathbf{w}-\mathbf{v}_u\|}{\delta_u}<t$ and
  $\mathbf{w}\in \bigcap_{u\in
    T}\mathbf{C}\left(\mathbf{v}_u,\tilde{t}\delta_u\right)$. Therefore,
  $\tau$ also has the property that
  \begin{equation}
    \label{eq:Lipschitz-ext:7}
    \tau<t\quad\iff\quad
    \bigcap_{u\in T}\mathbf{O}\left(\mathbf{v}_u,t\delta_u\right)
    \neq\emptyset. 
  \end{equation}
  According to \cref{eq:Lip-Ball-Intersc:4}, this implies that
  $\bigcap_{u\in
    T}\mathbf{C}\left(\mathbf{v}_u,\tau\delta_u\right)=\{\mathbf{y}\}$
  for some ${\mathbf{y}\in \R^n}$. Since
  $\bigcap_{u\in T}\mathbf{O}\left(\mathbf{v}_u,\tau\delta_u\right)
  =\emptyset$, it also implies that the point $\mathbf{y}$ is closest
  to the set $\left\{\mathbf{v}_u:u\in T\right\}$.\medskip

  To show that the so constructed $\tau\geq 0$ and
  $\mathbf{y}\in \bigcap_{u\in
    T}\mathbf{C}(\mathbf{v}_u,\tau\delta_u)$ are as in
  \cref{eq:Kirszbraun-Intersc-v1:1}, let
  \[
    C=\left\{u\in T:
      \|\mathbf{y}-\mathbf{v}_u\|=\tau\delta_u\right\}\quad\text{and}\quad
    D=\left\{u\in T:
      \|\mathbf{y}-\mathbf{v}_u\|<\tau\delta_u\right\}=T\setminus C.
  \]  
  It now follows from \cref{eq:Lipschitz-ext:7} that
  $C\neq \emptyset$. For the sake of contradiction, let us assume that
  $\mathbf{y}\notin \conv(\left\{\mathbf{v}_u:u\in
    C\right\})=A$. Next, let ${\mathbf{a}\in A}$ be the closest point
  to $\mathbf{y}$. Then by \cref{eq:Lipschitz-ext:11},
  ${\|\mathbf{v}_u-\mathbf{a}\|<
    \|\mathbf{v}_u-\mathbf{y}\|=\tau\delta_u}$ for every $u\in
  C$. Hence,
  $\bigcap_{u\in C}\mathbf{O}(\mathbf{v}_u,\tau\delta_n)\neq
  \emptyset$ and from \cref{eq:Lipschitz-ext:7} we obtain that
  ${D\neq \emptyset}$, so that
  $\mathbf{y}\in \bigcap_{u\in D} \mathbf{O}(\mathbf{v}_u,
  \tau\delta_u)$. However, by \Cref{proposition-Lip-Ball-Intersc:3},
  we also have that
  $\mathbf{y}\in \bigcap_{u\in C} \mathbf{C}(\mathbf{v}_u,
  \tau\delta_u)=\overline{\bigcap_{u\in C} \mathbf{O}(\mathbf{v}_u,
    \tau\delta_u)}$. Thus, we finally get that
  $\bigcap_{u\in T} \mathbf{O}(\mathbf{v}_u, \tau\delta_u)\neq
  \emptyset$, which, according to \cref{eq:Lipschitz-ext:7}, is
  obviously impossible.
\end{proof}

The extension property of nonexpansive maps in
\Cref{theorem-KIP-ver-15:1} now follows as a simple application of
\Cref{lemma-Lipschitz-ext:1}.

\begin{proof}[Proof of\/ \Cref{theorem-KIP-ver-15:1}]
  Let $T\subseteq \R^m$ be a nonempty finite set, $\varphi:T\to \R^n$
  be a nonexpansive map and $\mathbf{x}\in \R^m\setminus T$. Next, let
  $\tau\geq 0$ and $\mathbf{y}\in \R^n$ be as in
  \Cref{lemma-Lipschitz-ext:1} applied with the set $T$, the points
  $\mathbf{v}_u=\varphi(\mathbf{u})$ and the positive numbers
  $\delta_u=\|\mathbf{x}-\mathbf{u}\|>0$, $\mathbf{u}\in T$.  Then by
  this lemma, the set
  $\left\{\mathbf{u}\in T:
    \|\mathbf{y}-\mathbf{v}_u\|=\tau\delta_u\right\}$ is
  nonempty. Hence, we can take an enumeration
  $\left\{\mathbf{u}_1,\dots, \mathbf{u}_k\right\}$ of this set and,
  for convenience, let us set $\mathbf{v}_i=\mathbf{v}_{u_i}$ for
  every $i\leq k$.  Thus, by \cref{eq:Kirszbraun-Intersc-v1:1},
  $\mathbf{y}\in \conv\left(\left\{\mathbf{v}_i:i\leq
      k\right\}\right)$ and we get that
  $\mathbf{y}=\sum_{i=1}^k \lambda_i \mathbf{v}_i$ for some
  coefficients $\lambda_i\geq 0$, $i\leq k$, with
  $\sum_{i=1}^k \lambda_i=1$. Accordingly, this now implies that
  \begin{align*}
    0&=2\left\|\sum_{i=1}^k \lambda_i(\mathbf{y}-\mathbf{v}_i)\right\|^2
       =\sum_{i,j=1}^k \lambda_i\lambda_j
       \left[2\left\langle\mathbf{y}-\mathbf{v}_i, 
       \mathbf{y}-\mathbf{v}_j\right\rangle\right]\\
     &= \sum_{i,j=1}^k
       \lambda_i\lambda_j\left[\|\mathbf{y}-\mathbf{v}_i\|^2 +
       \|\mathbf{y}-\mathbf{v}_j\|^2 - \|\mathbf{v}_i-\mathbf{v}_j\|^2
       \right]\\
     &\geq \sum_{i,j=1}^k
       \lambda_i\lambda_j\left[\tau^2\|\mathbf{x}-\mathbf{u}_i\|^2 +
       \tau^2\|\mathbf{x}-\mathbf{u}_j\|^2 -
       \|\mathbf{u}_i-\mathbf{u}_j\|^2
       \right]\\
     & =\sum_{i,j=1}^k
       \lambda_i\lambda_j\left[\tau^2\cdot 2\langle\mathbf{x}-\mathbf{u}_i,
       \mathbf{x}-\mathbf{u}_j\rangle +(\tau^2-1)
       \|\mathbf{u}_i-\mathbf{u}_j\|^2
       \right]\\
     &= 2\tau^2\left\| \sum_{i=1}^k
       \lambda_i(\mathbf{x}-\mathbf{u}_i)\right\|^2 +(\tau^2-1) \sum_{i,j=1}^k
       \lambda_i\lambda_j\|\mathbf{u}_i-\mathbf{u}_j\|^2.
  \end{align*}

  If $k=1$, then $\tau=0\leq 1$ because $\mathbf{x}\neq
  \mathbf{u}_1$. Otherwise, if $k>1$, then $\tau^2-1\leq0$ which is
  equivalent to $\tau\leq 1$. Thus, we get that
  $\|\mathbf{y}-\varphi(\mathbf{u})\|\leq \tau
  \|\mathbf{x}-\mathbf{u}\| \leq \|\mathbf{x}-\mathbf{u}\| $ for every
  $\mathbf{u}\in T$. Hence, we can define a nonexpansive extension
  $\psi:T\cup\{\mathbf{x}\}\to \R^n$ of $\varphi$ by
  $\psi(\mathbf{x})=\mathbf{y}$ and $\psi\uhr T=\varphi$. Since by
  \Cref{lemma-Lipschitz-ext:1} we also have that
  $\psi(\mathbf{x})=\mathbf{y}$ is a closest point to the subset
  $\varphi(T)$, the proof is complete.
\end{proof}

Going back to the special property of the extension in
\Cref{theorem-KIP-ver-15:1}, let us point out the following result,
obtained by Fej\'er \cite{zbMATH02604339} for the plane and extended
to all Hilbert spaces by Phelps \cite[Theorem 4]{MR104139}. Here, it
is stated for Euclidean spaces, but essentially the same proof is
valid for all Hilbert spaces.

\begin{theorem}
  \label{theorem-Lipschitz-ext:5}
  If $A\subseteq \R^n$ is a nonempty set, then the set of all closest
  points to $A$ is the closed convex hull\, $\overline{\conv(A)}$\/
  of\/ $A$.
\end{theorem}

\begin{proof}
  Whenever $\mathbf{y}\notin \overline{\conv(A)}$, let
  $\mathbf{v}\in \overline{\conv(A)}$ be the closest point to
  $\mathbf{y}$. Then by \cref{eq:Lipschitz-ext:11}, the point
  $\mathbf{v}$ is pointwise closer to $A$ than
  $\mathbf{y}$. Conversely, suppose that $\mathbf{y}\in\R^n$ is not a
  closest point to $A$. Then there exists a point $\mathbf{v}\in \R^n$
  that is pointwise closer to $A$ than $\mathbf{y}$, i.e.\ with the
  property that $\|\mathbf{v}-\mathbf{a}\|< \|\mathbf{y}-\mathbf{a}\|$
  for every $\mathbf{a}\in A$. Consider the set
  $S=\left\{\mathbf{x}\in\R^n: \|\mathbf{v}-\mathbf{x}\|\leq
    \|\mathbf{y}-\mathbf{x}\|\right\}$, which is a closed half-space
  containing $A$. Indeed,
  $\mathbf{y}-\mathbf{x}=(\mathbf{y}-\mathbf{v})-
  (\mathbf{x}-\mathbf{v})$ and by the properties of the Euclidean
  norm,
  \[
    0\leq \|\mathbf{y}-\mathbf{x}\|^2-\|\mathbf{v}-\mathbf{x}\|^2=
    \|\mathbf{y}-\mathbf{v}\|^2 -2\langle
    \mathbf{y}-\mathbf{v},\mathbf{v}-\mathbf{x}\rangle.
  \]
  Accordingly,
  $S=\left\{\mathbf{x}\in \R^n: 2\langle\mathbf{y}-\mathbf{v},
    \mathbf{x}-\mathbf{v}\rangle \leq
    \|\mathbf{y}-\mathbf{v}\|^2\right\}$.  Thus,
  $\overline{\conv(A)}\subseteq S$ because $S$ is convex and
  $A\subseteq S$. Moreover, $\mathbf{y}\notin \overline{\conv(A)}$
  because $\mathbf{y}\notin S$. The proof is complete.
\end{proof}

\Cref{theorem-Lipschitz-ext:5} not only clarifies
\Cref{theorem-KIP-ver-15:1} but also provides a natural explanation
for the following refinement of the Kirszbraun extension theorem
(\Cref{theorem-Lipschitz-ext:1}).

\begin{corollary}
  \label{corollary-KIP-ver-16:2}
  If $A\subseteq \R^m$ is a nonempty set, then each nonexpansive map
  ${\varphi\colon A\to \R^n}$ can be extended to a nonexpansive map
  $f:\R^m\to \overline{\conv(\varphi(A))}$.
\end{corollary}

\begin{proof}
  Let $A\subseteq \R^m$ be a nonempty set and
  ${\varphi\colon A\to \R^n}$ be a nonexpansive map.  According to
  \Cref{remark-Lip-Ball-Intersc:1}, the proof reduces to extensions at
  a single point rather than the entire space. So, take a point
  $\mathbf{x}\in \R^n\setminus A$. Next, apply
  \Cref{theorem-Lipschitz-ext:1} to extend $\varphi$ to a nonexpansive
  map $\psi_0:A\cup\{\mathbf{x}\}\to \R^n$. If
  $\psi_0(\mathbf{x})\in \overline{\conv(\varphi(A))}$, we can simply
  take $\psi=\psi_0$. Otherwise, if
  $\psi_0(\mathbf{x})\notin \overline{\conv(\varphi(A))}$, let
  $\mathbf{v}\in \overline{\conv(\varphi(A))}$ be the closest point to
  $\psi_0(\mathbf{x})$. Then by \cref{eq:Lipschitz-ext:11}, just like
  in the proof of the \Cref{theorem-Lipschitz-ext:5}, the point
  $\mathbf{v}$ is pointwise closer to $\varphi(A)$ than
  $\psi_0(\mathbf{x})$. In other words,
  $\|\mathbf{v}-\varphi(\mathbf{u})\|< \|\psi_0(\mathbf{x})-
  \varphi(\mathbf{u})\|$ for every $\mathbf{u}\in A$. Hence, we can
  define a nonexpansive extension
  $\psi:A\cup\{\mathbf{x}\}\to \overline{\conv(\varphi(A))}$ of
  $\varphi$ by $\psi(\mathbf{x})=\mathbf{v}$ and $\psi\uhr A=\varphi$.
\end{proof}

\section{The Kirszbraun-Valentine Extension Theorem}
\label{sec:kirszbr-valent-exten}

By an \emph{inner product space} we mean a vector space endowed with
an inner product. If this inner product induces a complete norm, then
we call the corresponding normed space a \emph{Hilbert space}. In this
regard, let us remark that not every norm is derived by an inner
product. In fact, a norm $\|\cdot\|$ on a vector space $H$ is
generated by an inner product precisely when it satisfies the
well-known \emph{Jordan-von Neumann condition} \cite{Jordan1935},
i.e.\ when it satisfies the \emph{parallelogram identity}
\[
  \|u+v\|^2+ \|u-v\|^2=2\|u\|^2 + 2\|v\|^2\quad \text{for every
    $u,v\in H$.}
\]

The Kirszbraun extension theorem (\Cref{theorem-Lipschitz-ext:1}) is
also valid in the setting of Hilbert spaces. This was shown by
Valentine in \cite{zbMATH03102006}, and is sometimes called the
\emph{Kirszbraun-Valentine extension theorem}.

\begin{theorem}
  \label{theorem-Lipschitz-ext:6}
  Let $H$ be a Hilbert space, $L$ be an inner product space and
  $A\subseteq L$. Then each nonexpansive map $\varphi:A\to H$ can be
  extended to a nonexpansive map $f:L\to H$.
\end{theorem}

Some alternative proofs and generalisations of
\Cref{theorem-Lipschitz-ext:6} were also obtained by other authors,
for instance by Mickle \cite{zbMATH03049820}, Minty
\cite{zbMATH03180487,zbMATH03304171} and Gr\"unbaum and Zarantonello
\cite{Gruenbaum1968}. They will be discussed in the next
sections.\medskip

Let us explicitly remark that in \cite{zbMATH03102006}, Valentine gave
a different proof of \Cref{theorem-Lipschitz-ext:1}, and next showed
how to extend this theorem to all Hilbert spaces. Here, we will
outline his idea of how to make the step from Euclidean to Hilbert
spaces. Crucial to the proof of \Cref{theorem-Lipschitz-ext:6} are, as
before, one-point extensions of nonexpansive maps defined on finite
sets. This will now be based on the classical result of {\v S}mulian
\cite{MR0002006} that a normed space $E$ is reflexive if and only if
every decreasing sequence of nonempty closed bounded convex subsets of
$E$ has a nonempty intersection. In the setting of a Hilbert space
$H$, we follow the simple proof of this result given in \cite[Lemma
2.1 and Theorem 2.2]{Granas1995}, see also \cite[Lemma (4.1) and
Theorem (4.2) in Chapter I]{MR1987179}.  In this proof, we will use
$\delta(C)=d(0,C)$ for the distance between the origin $0\in H$ and a
nonempty subset $C\subseteq H$, namely
\begin{equation}
  \label{eq:Lip-Ball-Intersc:8}
  \delta(C)   =\inf_{x\in C}\|x\|= \inf\left\{\delta>0:
    \mathbf{C}(0,\delta)\cap C\neq 
    \emptyset\right\}.
\end{equation}

\begin{lemma}
  \label{lemma-Lipschitz-ext:2}
  If\/ $\left\{C_n:n\in\N\right\}$ is a decreasing sequence of
  nonempty closed bounded convex subsets of a Hilbert space $H$, then
  there exists a unique point $p\in\bigcap_{n=1}^\infty C_n$ with the
  property that\/ $\|p\| = \sup_{n\in\N} \delta(C_n)$. 
\end{lemma}

\begin{proof}
  The sequence $\left\{\delta(C_n):n\in\N\right\}$ is increasing
  because $\left\{C_n:n\in\N\right\}$ is decreasing. Moreover, since
  $\delta(C_n) \leq \delta(C_1) + \operatorname{diam}(C_1)$ for every
  $n \in \mathbb{N}$, it is also convergent. Let
  $\delta=\lim_{n\to\infty}\delta(C_n) =\sup_{n\in\N}\delta(C_n)$, and
  let us consider another decreasing sequence of nonempty closed
  bounded convex subsets $K_n\subseteq C_n$, $n\in \N$, defined by
  \begin{equation}
    \label{eq:Lipschitz-ext:15}
    K_n=\left\{x\in C_n: \|x\|\leq \delta+\mfrac1n\right\}=
    \mathbf{C}\left(0,\delta+\mfrac1n\right)\cap C_n.
  \end{equation}
  Then $\diam(K_n)\to 0$ as $n\to \infty$. Indeed, for each $n\in\N$,
  take points $x_n,y_n\in K_n$ with
  $\diam(K_n)\leq \|x_n-y_n\|+\frac1n$. Since
  $\frac{x_n+y_n}2\in K_n\subseteq C_n$, it follows from
  \cref{eq:Lip-Ball-Intersc:8,eq:Lipschitz-ext:15} that
  $\|x_n\|,\|y_n\|,\frac{\|x_n+y_n\|}2\in
  \left[\delta(C_n),\delta+\frac1n\right]$. Furthermore, since
  $\delta=\lim_{n\to \infty}\delta(C_n)$, we also have that
  $\lim_{n\to \infty}\left\|x_n\right\|=\lim_{n\to \infty}\|y_n\|=
  \lim_{n\to \infty}\frac{\|x_n+y_n\|}2=\delta$. Therefore, by the
  parallelogram identity,
  $\lim_{n\to \infty}\diam(K_n)\leq \lim_{n\to
    \infty}\|x_n-y_n\|+\frac1n= 0$ because
  \[
    \lim_{n\to \infty}\|x_n-y_n\|^2= \lim_{n\to
      \infty}\left(2\|x_n\|^2+ 2\|y_n\|^2 - \|x_n+y_n\|^2\right)=
    4\delta^2-4\delta^2=0.
  \]
  Finally, since $\left\{K_n:n\in\N\right\}$ is a decreasing sequence
  of nonempty closed subsets and $H$ is complete (being Hilbert), by
  the Cantor intersection theorem, $\bigcap_{n=1}^\infty K_n$ is a
  singleton, i.e.\ $\bigcap_{n=1}^\infty K_n=\{p\}$ for
  some $p\in H$. This $p$ is as required.
\end{proof}

\Cref{lemma-Lipschitz-ext:2} represents a generalisation of the fact
that each nonempty closed convex subset of a Hilbert space is
Chebyshev. It also constitutes the essential part of the following key
property of closed bounded convex subsets of Hilbert spaces.

\begin{theorem}
  \label{theorem-Lipschitz-ext:7}
  Let $\mathcal{F}$ be a nonempty family of closed bounded convex sets
  in a Hilbert space $H$. If\/ $\mathcal{F}$ has the finite
  intersection property, then\, $\bigcap \mathcal{F}\neq \emptyset$.
\end{theorem}

\begin{proof}
  Suppose that the family $\mathcal{F}$ has the finite intersection
  property and fix an element $S\in \mathcal{F}$. Next, consider the
  collection $\Sigma$ of all finite subsets
  $\sigma\subseteq \mathcal{F}$ with $S\in \sigma$, and for each
  $\sigma\in \Sigma$ set $C_\sigma=\bigcap \sigma$ and
  $\delta_\sigma=\delta(C_\sigma)$, see
  \cref{eq:Lip-Ball-Intersc:8}. Then
  $\delta_\sigma\leq \delta(S)+\diam(S)$ for every $\sigma\in \Sigma$,
  so $\delta=\sup_{\sigma\in \Sigma}\delta_\sigma<+\infty$. Moreover,
  $\delta_{\sigma_1}\leq \delta_{\sigma_2}$ whenever
  $\sigma_1,\sigma_2\in \Sigma$ with $\sigma_1\subseteq \sigma_2$
  because $C_{\sigma_2}\subseteq C_{\sigma_1}$. Hence, there exists an
  increasing sequence $\left\{\sigma_n\right\}\subseteq \Sigma$ such
  that $\delta-\frac1n\leq\delta_{\sigma_n}\leq \delta$,
  $n\in\N$. Accordingly, $\left\{C_{\sigma_n}\right\}$ is a decreasing
  sequence of nonempty closed bounded convex sets with
  $\delta=\sup_{n\in \N}\delta_{\sigma_n}$. Therefore, by
  \Cref{lemma-Lipschitz-ext:2}, there exists a unique point
  $p\in \bigcap_{n=1}^\infty C_{\sigma_n}$ with $\|p\|=\delta$.  To
  show that $p\in \bigcap \mathcal{F}$, take any $F\in \mathcal{F}$
  and set $\tilde{\sigma}_n=\sigma_n\cup\{F\}\in \Sigma$, $n\in
  \N$. The resulting sequence
  $\left\{\tilde{\sigma}_n\right\} \subseteq \Sigma$ is also
  increasing and $\sup_{n\in\N} \delta_{\tilde{\sigma}_n}=\delta$
  because
  $\delta_{\sigma_n}\leq \delta_{\tilde{\sigma}_n}\leq \delta$,
  $n\in\N$. Hence, just like before, \Cref{lemma-Lipschitz-ext:2}
  implies that $\|q\|=\delta$ for some a point
  $q\in \bigcap_{n=1}^\infty C_{\tilde{\sigma}_n}$. However,
  $q\in \bigcap_{n=1}^\infty C_{\tilde{\sigma}_n}\subseteq
  \bigcap_{n=1}^\infty C_{\sigma_n}$ and $\|q\|=\delta=\|p\|$, while
  $p$ is the unique point with this property. Thus, $p=q\in F$ and the
  proof is complete.
\end{proof}

According to \Cref{theorem-Lipschitz-ext:9,} and
\Cref{proposition-Lip-Ball-Intersc:4}, the finite intersection
property in \Cref{theorem-Lipschitz-ext:7} shows that
\Cref{corollary-Kirszbraun-IP-ver-01:1} remains valid for any Hilbert
range. A similar statement reflecting this fact was also established
in \cite[Lemma 1]{Gruenbaum1968}.

\begin{corollary}
  \label{corollary-KIP-ver-05:1}
  Let $(X,d)$ be a symmetric space and $H$ be a Hilbert space such
  that for every finite subset $T\subseteq X$ and $p\in X\setminus T$,
  each nonexpansive map $\varphi:T\to H$ can be extended to a
  nonexpansive map $\psi:T\cup\{{p}\}\to H$. Then for every
  $A\subseteq X$, each nonexpansive map $\varphi:A\to H$ can be
  extended to a nonexpansive map $f:X\to H$.
\end{corollary}

The proof of the Kirszbraun-Valentine extension theorem
(\Cref{theorem-Lipschitz-ext:6}) now follows from
\Cref{corollary-KIP-ver-05:1} and the following consequence of
\Cref{theorem-Lipschitz-ext:2}.

\begin{corollary}
  \label{corollary-Lip-Ball-Intersc:1}
  Let $H$ be a Hilbert space, $L$ be an inner product space and
  $T\subseteq L$ be a nonempty finite set. Then for each point
  ${p\in L\setminus T}$, each nonexpansive map\/ $\varphi:T\to H$ can
  be extended to a nonexpansive map\/ $\psi:T\cup\{p\}\to H$.
\end{corollary}

\begin{proof}
  Take a nonexpansive map $\varphi:T\to H$ and a point
  $p\in L\setminus T$. Next, let $E$ be the linear subspace of $L$
  spanned over $T\cup\{p\}$, and $V$\,---\,the one spanned over
  $\varphi(T)$. Then both $E$ and $V$ are isometric to some Euclidean
  spaces. Hence, by \Cref{theorem-Lipschitz-ext:2}, $\varphi:T\to V$
  can be extended to a nonexpansive map
  $\psi:T\cup\{p\}\to V\subseteq H$.
\end{proof}

We conclude this section with the following two remarks.

\begin{remark}
  \label{remark-KIP-ver-08:1}
  Let $L$ be an inner product space, $H$ be a Hilbert space and
  $f:L\to H$ be a nonexpansive map. Then $f$ is $1$-Lipschitz and by
  \Cref{corollary-Lipschitz-ext:3}, it can be extended to a
  nonexpansive map over the completion $E$ of $L$. Since $E$ is a
  Hilbert space, this shows that it makes no difference whether
  \Cref{theorem-Lipschitz-ext:6} and
  \Cref{corollary-Lip-Ball-Intersc:1} are formulated for an inner
  product domain or a Hilbert one.
\end{remark}

\begin{remark}
  \label{remark-KIP-ver-16:1}
  As in the previous section, the extension $f:L\to H$ in
  \Cref{theorem-Lipschitz-ext:6} can be constructed in such a way that
  $f(L)\subseteq \overline{\conv(\varphi(A))}$. This follows easily by
  applying the same arguments as in the proof of
  \Cref{corollary-KIP-ver-16:2}.
\end{remark}

\section{Schoenberg's Solution of Menger's Embedding Problem}
\label{sec:scho-solut-meng}

In 1928 and his subsequent refined paper in 1931, Karl Menger
\cite{MR1512479,zbMATH02561510} presented new axioms for geometry
based on the notion of distance, and provided necessary and sufficient
conditions for arbitrary sets to be ``congruent'' to subsets of
Euclidean spaces. This was achieved in the setting of a
\emph{congruence system} $(\mathfrak{S},\approx)$, where
$\mathfrak{S}$ is the collection of symmetric spaces, and the axioms
were defined in terms of a \emph{congruence relation}
$\left\{p,q\right\}\approx\left\{p',q'\right\}$ on the pairs of points
$p,q\in X$ and $p',q'\in Y$, for $X,Y\in\mathfrak{S}$.  In fact,
Menger defined a congruence system in a very abstract way, but the
essential example he provided was for symmetric spaces. Namely, for
such spaces $(X,d)$, $(Y,\rho)\in \mathfrak{S}$ and points $p,q\in X$
and $p',q'\in Y$, Menger's congruence relation
$\left\{p,q\right\}\approx\left\{p',q'\right\}$ was simply defined by
$d(p,q)=\rho(p',q')$. In the same setting, a \emph{congruent map}
$f:X\to Y$ between $(X,d)$, $(Y,\rho)\in \mathfrak{S}$ is now an
isometry. Thus, spaces $(X,d)$, $(Y,\rho)\in \mathfrak{S}$ are
\emph{congruent} if there exists a surjective isometry
$f:X\to Y$.\medskip

Menger's approach culminated into the concept of a \emph{congruence
  order}. Briefly, for a congruence system $\mathfrak{S}$, he defined
that a set $M\in \mathfrak{S}$ has a \emph{congruence order} $n$
provided that if $X\in \mathfrak{S}$ and every $n$ points of $X$ are
congruent to $n$ points of $M$, then $X$ itself is congruent to a
subset of $M$. The fundamental result proved by Menger is that the
Euclidean space $\R^n$ has a congruence order $n + 3$ but not $n + 2$
in the family $\mathfrak{S}$ of all symmetric spaces. The important
implication of this result is that in order to verify whether an
abstract symmetric space is congruent to a subset of $\R^n$, we only
need to verify congruence of each of its $n + 3$ point subsets. Based
on this, Menger completely resolved the embedding problem of symmetric
spaces in $\R^n$ by means of equations and inequalities involving
certain determinants composed by distances in such spaces. For an
extended discussion on Menger's congruence systems, the interested
reader is refereed to \cite{zbMATH06903590,zbMATH06641865}.\medskip

In a series of papers
\cite{Schoenberg1935,Schoenberg1937,Schoenberg1938}, Schoenberg
proposed an alternative approach to Menger's solution for embedding
finite sets into Euclidean spaces\,---\,one that is significantly
simpler, as it relies solely on definiteness of quadratic forms. His
starting point was the final section of Fr\'echet's paper
\cite{Frechet1935} on inner product spaces. Thus, in
\cite{Schoenberg1935}, he presented an elegant solution to the
following embedding problem.

\begin{problem}
  \label{problem-Lip-Ext-Embb:1}
  Let $\rho_{ij}=\rho_{ji}>0$, $i\neq j$ for $0\leq i,j\leq n$, be
  $\frac12 n(n+1)$ given positive numbers.  What are the necessary and
  sufficient conditions that they be the lengths of the edges of an
  $n$-simplex $\sigma=\{\mathbf{v}_0, \dots, \mathbf{v}_n\}$ in
  $\R^n$? More general, what are the conditions that they be the
  lengths of the edges of an $n$-``simplex''
  $\sigma=\{\mathbf{v}_0, \dots, \mathbf{v}_n\}$ lying in an Euclidean
  space $\R^m$ $(1\leq m\leq n)$ but not in $\R^{m-1}$?
\end{problem}

Here, an $n$-``simplex'' means a \emph{degenerate} simplex, i.e.\ a
set of $(n+1)$-points in $\R^m$ which is not necessarily affinely
independent. Below, we will briefly discuss Schoenberg's solution of
\Cref{problem-Lip-Ext-Embb:1}.\medskip

A \emph{quadratic form} in $\R^n$ is a function $Q:\R^n\to \R$ of the
form
\[
  Q(\mathbf{x})=Q(x_1,\dots,x_n)=\sum_{i=1}^n\sum_{j=1}^n
  a_{ij}x_ix_j,\quad\text{where $a_{ij}\in \R$, $1\leq i,j\leq n$.}
\]
Clearly, in the above expression, we can always assume that
$a_{ij} = a_{ji}$ by replacing $a_{ij}$ and $a_{ji}$ with
$\frac{a_{ij} + a_{ji}}{2}$, if necessary. In this way, a quadratic
$Q(\mathbf{x})$ is uniquely determined by a symmetric $n \times n$
matrix
\[
  \mathbf{A}=
  \begin{pmatrix}
    a_{11}& a_{12} &\dots& a_{1n}\\
    a_{21}& a_{22} &\dots & a_{2n}\\
    \vdots& \vdots&\ddots&\vdots\\
    a_{n1}& a_{n2}& \dots &a_{nn}
  \end{pmatrix}.
\]
Hence, with respect to this matrix $\mathbf{A}$, a quadratic form
$Q(\mathbf{x})$ can simply be written as
$Q(\mathbf{x}) =\langle\mathbf{A}\mathbf{x},\mathbf{x}\rangle$ for
each $\mathbf{x}\in\R^n$.\medskip

A quadratic form $Q(\mathbf{x})$ is \emph{positive definite} if
$Q(\mathbf{x})>0$ for every $\mathbf{x}\neq \mathbf{0}$, and it is
\emph{positive semidefinite} if $Q(\mathbf{x})\geq 0$ for every
$\mathbf{x}\in\R^n$. The \emph{rank} of a quadratic form $Q:\R^n\to\R$
is the rank of the unique symmetric $n\times n$ matrix $\mathbf{A}$
for which
$Q(\mathbf{x})=\langle\mathbf{A}\mathbf{x},\mathbf{x}\rangle$.
Alternatively, the rank of $Q(\mathbf{x})$ is $m\leq n$ if the maximal
linear subspace of $\R^n$ on which $Q(\mathbf{x})$ will vanish is of
dimension $n-m$.\medskip

Another useful interpretation follows from the fact that for each
quadratic form $Q(\mathbf{x})$ there exists a change of basis (i.e.\ a
linear isomorphism) $\Omega:\R^n\to \R^n$ and numbers
$\lambda_1,\dots, \lambda_m\in \{-1,1\}$, $m\leq n$, such that
\begin{equation}
  \label{eq:Lip-Ext-Embb:2}
  Q\left(\Omega\mathbf{y}\right)= \lambda_1y_1^2+\dots
  +\lambda_my_m^2\quad \text{for every\, $\mathbf{y}\in \R^n$.}
\end{equation}
Such a basis for $\R^n$ is called a \emph{canonical basis} for
$Q(\mathbf{x})$, and the expression in \cref{eq:Lip-Ext-Embb:2} is
called the \emph{canonical form} of $Q(\mathbf{x})$. In terms of this
canonical form, $m$ is the rank of $Q(\mathbf{x})$ and it is positive
semidefinite if all coefficients $\lambda_1,\dots, \lambda_m$ are
positive, i.e.\ equal to $1$. In particular, $Q(\mathbf{x})$ is
positive definite if it is positive semidefinite and $m=n$. \medskip

The following theorem was proved by Schoenberg in \cite[Theorem
1]{Schoenberg1935}.

\begin{theorem}
  \label{theorem-Lip-Ext-Embb:1}
  Let $(A,\rho)$ be a symmetric space consisting of $n+1$ points
  $a_0,a_1,\dots, a_n$, and $\rho_{ij}=\rho(a_i,a_j)$,
  $0\leq i,j\leq n$, be the distances between these points. Then
  $(A,\rho)$ can be isometrically embedded in\, $\R^{n}$ if and only
  if the quadratic form
  \begin{equation}
    \label{eq:Lip-Ext-Embb:9}
    \left\{
      \begin{aligned}
        Q(x_1,\dots,
        x_n)&=\mfrac12\sum_{i,j=1}^n\left(\rho_{0i}^2+\rho_{0j}^2-
          \rho_{ij}^2\right)x_ix_j\\
        &=\sum_{i=1}^n \rho_{0i}^2x_i^2 +\sum_{1\leq i<j\leq
          n}\left(\rho_{0i}^2+\rho_{0j}^2- \rho_{ij}^2\right)x_ix_j
      \end{aligned}
    \right.
  \end{equation}
  is positive semidefinite. Moreover, $(A,\rho)$ can be embedded
  isometrically in\, $\R^m$, but not in\, $\R^{m-1}$, if the rank of
  this quadratic form is $m$.
\end{theorem}

\begin{proof}
  Let
  $\sigma=\left\{\mathbf{v}_0, \dots, \mathbf{v}_n\right\}\subseteq
  \R^n$ and $\rho_{ij}=\|\mathbf{v}_i-\mathbf{v}_j\|$ for every
  $i,j\leq n$. Since any translation of $\R^n$ is an isometry, we can
  assume that $\mathbf{v}_0=\mathbf{0}$ is the origin of $\R^n$ so
  that
  $\rho_{i0}=\rho_{0i} =\left\|\mathbf{v}_i-\mathbf{v}_0\right\|
  =\|\mathbf{v}_i\| $ for every $1\leq i\leq n$. Next, define a linear
  map $\Lambda:\R^n\to \R^n$ by
  $\Lambda\mathbf{x}=x_1\mathbf{v}_1+\dots +x_n\mathbf{v}_n$ for every
  $\mathbf{x}=(x_1,\dots, x_n)\in \R^n$. If $Q(\mathbf{x})$ is as in
  \cref{eq:Lip-Ext-Embb:9}, then
  \[
    \|\Lambda \mathbf{x}\|^2 =
    \langle\Lambda\mathbf{x},\Lambda\mathbf{x}\rangle =\left\langle
      \sum_{i=1}^nx_i\mathbf{v}_i, \sum_{j=1}^nx_j\mathbf{v}_j
    \right\rangle
    =\sum_{i,j=1}^n\left\langle\mathbf{v}_i,\mathbf{v}_j\right\rangle
    x_ix_j =Q(\mathbf{x}),
  \]
  because for every $1\leq i,j\leq n$ we have that
  \[
    \rho_{ij}^2= \left\langle\mathbf{v}_i-\mathbf{v}_j,
    \mathbf{v}_i-\mathbf{v}_j\right\rangle =\left\langle
    \mathbf{v}_i,\mathbf{v}_i\right\rangle +\langle
    \mathbf{v}_j,\mathbf{v}_j\rangle
    -2\langle\mathbf{v}_i,\mathbf{v}_j\rangle =\rho_{0i}^2 +
    \rho_{0j}^2 - 2\langle\mathbf{v}_i,\mathbf{v}_j\rangle.
  \]
  Therefore, $Q(\mathbf{x})=\|\Lambda\mathbf{x}\|^2\geq 0$ is positive
  semidefinite. Moreover, since $\Lambda$ is linear, $\R^n$ is the
  direct sum of the \emph{image} $\Lambda\left(\R^n\right)$ of
  $\Lambda$ and its \emph{kernel}
  $\left\{\mathbf{x}\in \R^n:
    \Lambda\mathbf{x}=\mathbf{0}\right\}$. Hence, the rank of
  $Q(\mathbf{x})$ is the dimension of the linear subspace
  $\Lambda\left(\R^n\right)\subseteq\R^n$.\smallskip

  Conversely, let the quadratic form $Q(\mathbf{x})=Q(x_1,\dots, x_n)$
  in \cref{eq:Lip-Ext-Embb:9} be positive semidefinite, and its rank
  be $m\leq n$. Next, let $\Omega:\R^n\to \R^n$ be the linear
  isomorphism as in \cref{eq:Lip-Ext-Embb:2}, i.e.\ with the property
  that $Q\left(\Omega\mathbf{y}\right)= y_1^2+\dots +y_m^2$ for every
  $\mathbf{y}\in \R^n$. Also, for convenience, let $\Pi:\R^n\to \R^m$
  be the orthogonal projection $\Pi(\mathbf{y})=(y_1,\dots, y_m)$ for
  ${\mathbf{y}=(y_1,\dots, y_n)\in \R^n}$, so that the representation
  in \cref{eq:Lip-Ext-Embb:2} is now simply given by
  \begin{equation}
    \label{eq:Lip-Ext-Embb:5}
    Q(\Omega\mathbf{y})=y_1^2+\dots +y_m^2=
    \left\|\Pi\mathbf{y}\right\|^2\quad 
    \text{for every\, $\mathbf{y}=(y_1,\dots, y_n)\in \R^n$.}
  \end{equation}
  Finally, we define $\mathbf{y}_i=\Omega^{-1}\mathbf{e}_i$ for each
  $1\leq i\leq n$, where $\mathbf{e}_1,\dots, \mathbf{e}_n\in \R^n$
  are the standard basic unit vectors. Then, taking
  $\mathbf{v}_0=\mathbf{0}\in \R^m$ and
  $\mathbf{v}_i=\Pi\mathbf{y}_i$, $1\leq i\leq n$, it
  follows from \cref{eq:Lip-Ext-Embb:9,eq:Lip-Ext-Embb:5} that for
  each $i\leq n$ and $j\neq i$,
  \begin{align*}
    \|\mathbf{v}_i-\mathbf{v}_0\|^2
    &=\|\mathbf{v}_i\|^2=
      \left\|\Pi\mathbf{y}_i\right\|^2
      = Q(\Omega\mathbf{y}_i)=
      Q(\mathbf{e}_i)=\rho_{0i}^2,\qquad
      \text{and} \\
    \|\mathbf{v}_i-\mathbf{v}_j\|^2
    &= \left\|\Pi(\mathbf{y}_i-\mathbf{y}_j)\right\|^2
      =
      Q\left(\Omega(\mathbf{y}_i-\mathbf{y}_j)\right)
      =
      Q(\mathbf{e}_i-\mathbf{e}_j)=\rho_{ij}^2.
  \end{align*}
  Accordingly, the set
  $\sigma=\left\{\mathbf{v}_0,\dots ,\mathbf{v}_n\right\} \subseteq
  \R^m$ is as required.
\end{proof}

\Cref{theorem-Lip-Ext-Embb:1} gives a necessary and sufficient
condition for isometric embeddings of symmetric spaces into Hilbert
spaces, see \cite{Schoenberg1937,Schoenberg1938} and \cite[Lemma
2.3]{MR0461107}. Accordingly, the Kirszbraun-Valentine extension
theorem (\Cref{theorem-Lipschitz-ext:6}) is also valid for such
symmetric spaces. Here, we are interested in the quadratic form in
\cref{eq:Lip-Ext-Embb:9} and its role in the extension problem of
nonexpansive maps. In view of this, in some cases it will be more
convenient to represent this quadratic form not in terms of
coordinates, but by using the interpretation that $Y^X$ is the set of
all maps from $X$ to $Y$. Accordingly, if $A$ is as in
\Cref{theorem-Lip-Ext-Embb:1} and $T=A\setminus\{a_0\}$, then the
doubled quadratic form in \cref{eq:Lip-Ext-Embb:9} can be simply
written as
\[
  Q(\xi)=\sum_{(s,t)\in T^2}\left[\rho_{0s}^2+ \rho_{0t}^2-
    \rho_{st}^2\right]\xi(s)\xi(t)\quad \text{for every $\xi\in \R^T$.}
\]

To bring the above quadratic form $Q(\xi)$ into line with the
extension problem for nonexpansive maps, as in
\Cref{theorem-Lip-Ext-Embb:1}, for a symmetric space
$(X,\rho)$ we will use $\rho_{xy}$ to denote the distance
between points $x,y\in X$, i.e.\
\begin{equation}
  \label{eq:KIP-ver-17:2}
  \rho_{xy}=\rho(x,y)\quad \text{for every $x,y\in X$.}
\end{equation}
Furthermore, for a nonempty finite set $T\subseteq X$ and a point
$p\in X\setminus T$, we will denote the quadratic form $Q(\xi)$ as
$Q_{(p,T)}(\xi)$ to explicitly emphasise that it is associated
specifically to the pair $(p,T)$:
\[
  Q_{(p,T)}(\xi)=\sum_{(s,t)\in T^2}\left[\rho_{ps}^2+ \rho_{pt}^2-
    \rho_{st}^2\right]\xi(s)\xi(t)\quad \text{for every
    $\xi\in \R^T$.}
\]
Finally, for a nonempty (finite) set $T$, let $\R^T_+$ be the
\emph{positive cone} of $\R^T$, i.e.\ the subset of all functions
$\xi:T\to [0,+\infty)$. Essentially, in this context, Gr\"unbaum and
Zarantonello \cite{Gruenbaum1968} relaxed the condition in
\Cref{theorem-Lip-Ext-Embb:1} imposed on the quadratic form in
\cref{eq:Lip-Ext-Embb:9}. Namely, in view of the extension problem of
nonexpansive maps, they did not require $Q_{(p,T)}$ to be positive
semidefinite, but instead to have the property that
\begin{equation}
  \label{eq:KIP-ver-17:1}
  Q_{(p,T)}(\xi)=\sum_{(s,t)\in T^2}\left[\rho_{ps}^2+
    \rho_{pt}^2-
    \rho_{st}^2\right]\xi(s)\xi(t)\geq 0\quad \text{for every $\xi\in
    \R^T_+$.} 
\end{equation}

In a slightly different form, the following result was obtained in
\cite[Lemma 2]{Gruenbaum1968}.

\begin{theorem}
  \label{theorem-KIP-ver-17:1}
  Let $H$ be a Hilbert space and $(X,\rho)$ be a symmetric space such
  that for every nonempty finite set $T\subseteq X$ and
  $p\in X\setminus T$, the quadratic form $Q_{(p,T)}$ satisfies the
  condition in \cref{eq:KIP-ver-17:1}. Then for every $A\subseteq X$,
  each nonexpansive map $\varphi:A\to H$ can be extended to a
  nonexpansive map $f:X\to H$.
\end{theorem}

\begin{proof}
  Let $T\subseteq X$ be a nonempty finite set, $p\in X\setminus T$ and
  $\varphi:T\to H$ be a nonexpansive map. Then the linear subspace
  $V\subseteq H$ spanned over $\varphi(T)$ is isometric to some
  Euclidean space $\R^n$ so that $v_t=\varphi(t)\in V\cong\R^n$ and
  $\rho_{pt}=\rho(p,t)>0$ for every $t\in T$. Therefore, as before,
  let $\tau\ge 0$ and
  $y\in \bigcap_{t\in T} \mathbf{C}\left(v_t,\tau\rho_{pt}\right)$ be
  as those in \Cref{lemma-Lipschitz-ext:1}. Next, by the same lemma,
  let $\left\{t_1,\dots, t_k\right\}$ be an enumeration of the set
  $\left\{t\in T: \|y-v_t\|=\tau\rho_{pt}\right\}$ and $v_i=v_{t_i}$
  for every $i\leq k$. Thus, by \cref{eq:Kirszbraun-Intersc-v1:1}, we
  now have that $y=\sum_{i=1}^k \lambda_i v_i$ for some coefficients
  $\lambda_i\geq 0$, $i\leq k$, with $\sum_{i=1}^k
  \lambda_i=1$. Accordingly, we can apply the fact that
  $\varphi:T\to H$ is nonexpansive and the condition in
  \cref{eq:KIP-ver-17:1} to conclude that
  \begin{align*}
    0&=2\left\|\sum_{i=1}^k \lambda_i(y-v_i)\right\|^2
       =\sum_{i,j=1}^k \lambda_i\lambda_j\cdot 2\left\langle y-v_i,
       y-v_j\right\rangle\\
     &= \sum_{i,j=1}^k
       \lambda_i\lambda_j\left[\|y-v_i\|^2 +
       \|y-v_j\|^2 - \|v_i-v_j\|^2
       \right]\\
     &\geq \sum_{i,j=1}^k
       \lambda_i\lambda_j\left[\tau^2\rho_{pt_i}^2 +
       \tau^2\rho_{pt_j}^2 - \rho_{t_it_j}^2
       \right]\\
     &\geq\sum_{i,j=1}^k
       \lambda_i\lambda_j\left[\tau^2\rho_{pt_i}^2 +
       \tau^2\rho_{pt_j}^2 - \rho_{pt_i}^2- \rho_{pt_j}^2
       \right]\\
      &=\left(\tau^2-1\right)\cdot\sum_{i,j=1}^k
       \lambda_i\lambda_j\left[\rho_{pt_i}^2 +
       \rho_{pt_j}^2\right].
  \end{align*}

  It is obvious that this is possible only when $\tau^2-1\leq0$, or in
  other words, when $\tau\leq 1$. Consequently,
  $\|y-\varphi(t)\|\leq \rho_{pt}=\rho(p,t)$ for every $t\in T$, and a
  nonexpansive extension $\psi:T\cup\{p\}\to V\subseteq H$ of
  $\varphi$ is simply defined by $\psi(p)=y$ and $\psi\uhr
  T=\varphi$. The rest of the proof follows from
  \Cref{corollary-KIP-ver-05:1}.
\end{proof}

We conclude this section with some special cases of
\Cref{theorem-KIP-ver-17:1}, which are naturally related to
\Cref{theorem-Lipschitz-ext:6,theorem-Lip-Ext-Embb:1}.  In
\cite{zbMATH03049820}, Mickle proposed a generalisation of the
Kirszbraun-Valentine extension theorem
(\Cref{theorem-Lipschitz-ext:6}) based on the following
considerations. For a metric space $(X,d)$, let $\mathpzc{P}(X)$ be
the set of all functions $\omega:[0,+\infty)\to [0,+\infty)$ such that
\addtocounter{equation}{1}
\begin{enumerate}[label=(\theequation\alph*)]
\item\label{item:KIP-ver-10:1} $\omega(t)=0$ precisely when $t=0$,\ \
  and\smallskip
\item\label{item:KIP-ver-10:2}
  $Q_{(p,T)}^\omega(\xi)=\medop\sum_{(s,t)\in
    T^2}\left[\omega^2(d_{ps})+ \omega^2(d_{pt})-
    \omega^2(d_{st})\right]\xi(s)\xi(t)\geq0$,
  $\xi\in \R^T_+$, for every nonempty finite set $T\subseteq X$ and
  $p\in X\setminus T$.
\end{enumerate}

In this setting, Mickle proved the following extension theorem in
\cite[Theorem]{zbMATH03049820}, which is now a simple consequence of
\Cref{theorem-KIP-ver-17:1}.

\begin{corollary}
  \label{corollary-KIP-ver-24:1}
  Let $H$ be a Hilbert space, $(X,d)$ be a metric space,
  $\omega\in \mathpzc{P}(X)$ and $A\subseteq X$. If\/ $\varphi:A\to H$
  is a map such that $\|\varphi(x)-\varphi(y)\|\leq \omega(d_{xy})$
  for every ${x,y\in A}$, then $\varphi$ can be extended to a map
  $f:X\to H$ such that $\|f(x)-f(y)\|\leq \omega(d_{xy})$
  for every $x,y\in X$.
\end{corollary}

\begin{proof}
  Consider the function $\rho:X\times X\to [0,+\infty)$ defined by
  $\rho(x,y)= \omega(d(x,y))$ for every $x,y\in X$. Then by
  \ref{item:KIP-ver-10:1}, $\rho$ is a symmetric on $X$. Moreover, not
  only is the quadratic form $Q_{(p,T)}$ with respect to $\rho$ the
  same as the quadratic form $Q_{(p,T)}^\omega$ with respect to $d$
  and the function $\omega\in\mathpzc{P}(X)$, but also the condition
  in \ref{item:KIP-ver-10:2} is the same as that in
  \cref{eq:KIP-ver-17:1}. Furthermore, if $\varphi:A\to H$ is a map
  such that
  ${\|\varphi(x)-\varphi(y)\|\leq \omega(d_{xy})}$ for
  every ${x,y\in A}$, then $\varphi$ is nonexpansive with respect to
  $\rho$. Hence, by \Cref{theorem-KIP-ver-17:1}, $\varphi$ can be
  extended to a map $f:X\to H$ that is also nonexpansive with respect
  to $\rho$. This $f$ is clearly as required because
  $\rho_{xy}=\omega(d_{xy})$ for every $x,y\in X$.
\end{proof}

The proof of \Cref{corollary-KIP-ver-24:1} in \cite{zbMATH03049820} is
technically demanding and follows Schoenberg's approach in
\cite[Theorem 2]{zbMATH03084778} to establish the Kirszbraun
intersection property in \Cref{theorem-Lip-Ball-Intersc:1}, see also
\Cref{lemma-Lipschitz-ext:1}. Furthermore, in \cite{zbMATH03049820}
this extension property is formulated for an Euclidean space and
$(X,d)$ is assumed to be separable. At the end of his paper, however,
Mickle added a remark that his theorem is valid for any metric domain
and any Hilbert range, which was pointed out to him by the referee. In
contrast, our proof of \Cref{corollary-KIP-ver-24:1} is not only
simpler, without repeating previous proofs, but also does not use the
hypothesis in \cite{zbMATH03049820} that $\omega\in \mathpzc{P}(X)$
must be a continuous function and the quadratic form
$Q_{(p,T)}^\omega$ must be positive semidefinite.\medskip

We proceed with another consequence of \Cref{theorem-KIP-ver-17:1}. To
this end, suppose that $(X,\rho)$ is a symmetric space such that
$\rho^2$ is a metric on $X$, i.e.\ that
\[
  \left[\rho(s,t)\right]^2\leq \left[\rho(s,p)\right]^2 +
  \left[\rho(p,t)\right]^2 \quad \text{for every $s,t,p\in X$.}
\]
In terms of our convention in \cref{eq:KIP-ver-17:2}, this means that
$\rho_{ps}^2+ \rho_{pt}^2- \rho_{st}^2\geq0$ for every ${s,t,p\in
  X}$. Accordingly, in this case the symmetric $\rho$ satisfies
\cref{eq:KIP-ver-17:1} and we obtain the following immediate
consequence.

\begin{corollary}
  \label{corollary-KIP-ver-17:1}
  Let $H$ be a Hilbert space and $(X,\rho)$ be a symmetric space such
  that $\rho^2$ is a metric on $X$. Then for every $A\subseteq X$,
  each nonexpansive map $\varphi:A\to H$ can be extended to a
  nonexpansive map $f:X\to H$.
\end{corollary}

Each Hilbert space is isomorphic to the space $\ell_2 (\mathcal{A})$
for some set $\mathcal{A}$, i.e.\ to the space of all functions
$y : \mathcal{A} \to \R$ with
${\sum _{\alpha \in \mathcal{A}} [y (\alpha)]^2 < \infty}$, where the
linear operations are defined pointwise and the norm is
$\| y \| = \sqrt{\sum _{\alpha \in \mathcal{A}} [y (\alpha)]^2}$. The
following result, based on a simple construction in \cite[Corollary
1.15]{Benyamini2000}, is complementary to
\Cref{corollary-KIP-ver-17:1} and shows that the extension property
there is also a consequence of the Kirszbraun-Valentine extension
theorem (\Cref{theorem-Lipschitz-ext:6}).

\begin{corollary}
  \label{corollary-KIP-ver-24:2}
  Let $H$ be a Hilbert space, $(X,\rho)$ be a symmetric space such
  that $\rho^2$ is a metric on $X$ and $A\subseteq X$ be a nonempty
  set. Then for each $p\in X\setminus A$ there exists an injective map
  $\eta_p:A\to \ell_2(A)$ such that
  \begin{equation}
    \label{eq:Lip-Ext-Embb:13}
    \|\eta_p(a)\|=\rho_{pa}\ \ \text{and}\ \ \rho_{ab}\leq
    \|\eta_p(a)-\eta_p(b)\|\quad\text{for every $a,b\in A$.} 
  \end{equation}
  In particular, each nonexpansive map $\varphi:A\to H$ can be
  extended to a nonexpansive map $f:X\to H$.
\end{corollary}

\begin{proof}
  We follow a construction in the proof of \cite[Corollary
  1.15]{Benyamini2000}. Namely, take a point $p\in X\setminus
  A$. Next, for each $a\in A$, define a function
  $\eta_p(a)[\cdot]:A\to \R$ by
  \[
    \eta_p(a)[x]=
    \begin{cases}
      \rho_{px} &\text{if $x=a$,\ \ and } \\
      0 &\text{if $x\neq a$.}
    \end{cases}
  \]
  Evidently, this defines an injective map $\eta_p:A\to \ell_2(A)$
  such that $\|\eta_p(a)\|=\rho_{pa}$ for every $a\in A$. Similarly,
  since $\rho^2$ is a metric on $X$, we also have that
  \[
    \rho_{ab}\leq \sqrt{\rho_{ap}^2 + \rho_{pb}^2}=
    \|\eta_p(a)-\eta_p(b)\|\quad \text{for every $a,b\in A$.}
  \]
  This shows that $\eta_p:A\to \ell_2(A)$ is as in
  \cref{eq:Lip-Ext-Embb:13}. To see that this implies the rest of the
  statement, let $\varphi:A\to H$ be a nonexpansive map and
  $p\in X\setminus A$. Next, for convenience, let
  $A_p=\eta_p(A)\subseteq \ell_2(A)$. Then, since $\eta_p:A\to A_p$ is
  an injective map, we can define a map $h:A_p\to H$ by
  $h\left(\eta_p(a)\right)=\varphi(a)$ for every $a\in A$, i.e.\ for
  which the following diagram is commutative:
  \begin{center}
    \vspace{-.5em}
    \begin{tikzcd}
      A \arrow[d, swap, "\eta_p"] \arrow[r, "\varphi"] &  H\\
      A_p \arrow[ur, dashed, swap, "h"] &
    \end{tikzcd}
  \end{center}
  Furthermore, it now follows from \cref{eq:Lip-Ext-Embb:13} that the
  map $h:A_p\to H$ is also nonexpansive because
  $\left\|h\left(\eta_p(a)\right)-
    h\left(\eta_p(b)\right)\right\|=\|\varphi(a)-\varphi(b)\|\leq
  \rho_{ab}\leq \|\eta_p(a)-\eta_p(b)\|$ for every $a,b\in
  A$. Accordingly, by \Cref{theorem-Lipschitz-ext:6}, $h$ can be
  extended to a nonexpansive map $g:\ell_2(A)\to H$. Finally, we can
  define an extension $\psi:A\cup\{p\}\to H$ of $\varphi$ by
  $\psi\uhr A=\varphi$ and $\psi(p)=g(0)$. By
  \cref{eq:Lip-Ext-Embb:13}, this extension is also nonexpansive
  because
  \[
    \|\psi(a)-\psi(p)\|= \left\|g\left(\eta_p(a)\right)-
      g(0)\right\|\leq \|\eta_p(a)-0\|\leq \|\eta_p(a)\|=\rho_{pa}.
  \]
  Hence, by \Cref{theorem-Lipschitz-ext:9}, the proof is complete.
\end{proof}

We conclude this section with the following remarks regarding the
Gr\"unbaum-Za\-rantonello condition in \cref{eq:KIP-ver-17:1}.

\begin{remark}
  \label{remark-KIP-ver-28:1}
  An \emph{ordered vector
    space} is a vector space $E$ endowed with a partial order $\leq$
  such that for every $u, v \in E$ with $u\leq v$,
  \begin{enumerate}[label=(\roman*)]
  \item $u+x\leq v+x$\quad for every $x\in E$,
  \item $\alpha \cdot u\leq \alpha \cdot v$\quad\ \ \ for every
    $\alpha>0$.
  \end{enumerate}
  In such a space $E$, the set $\{v\in E: v\geq 0\}$ is called the
  \emph{positive cone} of $E$ and denoted by $E^+$ or $E_+$.  An
  ordered vector space $(E,\leq)$ is a \emph{Riesz space}, or a
  \emph{vector lattice}, if it is also a lattice, i.e.\ if each pair
  of elements $u,v\in E$ has a supremum $u\vee v=\sup\{u,v\}\in E$ and
  an infimum $u\wedge v=\inf\{u,v\}\in E$. Finally, let us recall that
  a map $\Phi:E\to V$ between Riesz spaces is \emph{positive} if
  ${\Phi(E^+) \subseteq V^+}$.  The quadratic form
  $Q_{(p,T)}:\R^T\to \R$ in \cref{eq:KIP-ver-17:1} is a function that
  preserves the positive cone $\R^T_+$ of the vector space $\R^T$ in
  the sense that it maps it into the positive cone
  ${\R_+=[0,+\infty)}$ of the real line. Hence, in these terms, the
  Gr\"unbaum-Zarantonello condition in \cref{eq:KIP-ver-17:1} simply
  means that the quadratic form $Q_{(p,T)}:\R^T\to \R$ is a positive
  function between Riesz spaces.
\end{remark}

\begin{remark}
  \label{remark-KIP-draft-04:1}
  Let $(X,d)$ be a metric space and $\rho=\sqrt{d}$. Then $(X,\rho)$
  is a symmetric space such that $\rho^2=d$ is a metric on $X$. Hence,
  an interpretation of \Cref{corollary-KIP-ver-17:1} is that for a
  Hilbert space $H$, a metric space $(X,d)$ and $A\subseteq X$, each
  $\sqrt{d}$-nonexpansive map $\varphi:A\to H$ can be extended to a
  $\sqrt{d}$-nonexpansive map $f:X\to H$. Here, evidently, it is
  assumed from the context that ``\emph{$\sqrt{d}$-nonexpansive}''
  means a map which is nonexpansive with respect to the symmetric
  $\rho=\sqrt{d}$ on $X$. Similarly, the same interpretation applies
  to \Cref{corollary-KIP-ver-24:2} as well. 
\end{remark}

\section{Fox's Inequality and H\"older Continuity}
\label{sec:foxs-ineq-hold}

Let $(X,d)$ and $(Y,\rho)$ be metric spaces and $\alpha>0$. A map
$f:X\to Y$ is called \emph{$\alpha$-H\"older continuous}, or
\emph{H\"older continuous with exponent $\alpha$}, if there exists
$K\geq 0$ such that
\[
  \rho(f(p),f(q))\leq K \left[d(p,q)\right]^\alpha\quad \text{for
    every $p,q\in X$.}
\]
Here, the constant $K$ is often called the \emph{H\"older coefficient}
of the map $f:X\to Y$. In these terms, each $1$-H\"older continuous
map is $K$-Lipschitz. Another special case is when the H\"older
coefficient is $1$, i.e.\ when $\rho(f(p),f(q))\leq [d(p,q)]^\alpha$
for every $p,q\in X$. In this case, we will simply say that $f$ is
\emph{$\alpha$-nonexpansive}. Thus, if $(Y,\rho)$ is a linear normed
space $(E,\|\cdot\|)$, then a map $f:X\to E$ is $\alpha$-H\"older
continuous with a coefficient $K>0$ if and only if the map
$\frac1{K} f:X\to E$ is $\alpha$-nonexpansive.\medskip

Another natural interpretation of $\alpha$-nonexpansive maps is given
by the following construction. For a metric space $(X,d)$ and
$\alpha>0$, we will associate the function
${d^\alpha:X\times X\to \R}$, defined by
\begin{equation}
  \label{eq:KIP-ver-09:1}
  d^\alpha(p,q)=\left[d(p,q)\right]^\alpha\quad \text{for every $p,q\in X$.}
\end{equation}
Evidently, $d^\alpha$ is always a symmetric on $X$. Hence, as in
\Cref{remark-KIP-draft-04:1}, we have the following simple
observation. 

\begin{proposition}
  \label{proposition-KIP-draft-04:1}
  For $\alpha>0$ and metric spaces $(X,d)$ and $(Y,\rho)$, a map
  $f:X\to Y$ is $\alpha$-nonexpansive if and only if it is
  $d^\alpha$-nonexpansive.
\end{proposition}

In general, the symmetric $d^\alpha:X\times X\to \R$ is not
necessarily a metric on $X$, but it is a metric in the special case
where $\alpha \leq 1$.  To consider this case, let us recall that a
function ${\omega:[0,+\infty)\to \R}$ is \emph{subadditive} if
\[
  \omega(s+t)\leq \omega(s)+ \omega(t)\quad \text{for every
    $s,t\geq 0$.}
\]
A function $\omega:[0,+\infty)\to [0,+\infty)$ is called
\emph{quasi-homogeneous} \cite{Rosenbaum1950} if
\begin{equation}
  \label{eq:Uniform-Cont-Ext:2}
  \omega(\lambda t)\leq \lambda\omega(t)\qquad \text{for every
    $\lambda\geq 1$ and $t\geq 0$.} 
\end{equation}

The quasi-homogeneous functions have the following natural
characterisation, see \cite[Theorem 1.4.3]{Rosenbaum1950}.

\begin{proposition}
  \label{proposition-Uniform-Cont-Ext:4}
  A function $\omega:[0,+\infty)\to [0,+\infty)$ is quasi-homogeneous
  if and only if the function\, $\frac{\omega(t)}{t}$, $t>0$, is
  decreasing.
\end{proposition}

\begin{proof}
  Follows from the fact that for $\lambda\geq 1$ and $t>0$, the
  condition in \cref{eq:Uniform-Cont-Ext:2} is equivalent to
  $\frac{\omega(\lambda t)}{\lambda t}\leq \frac{\omega(t)}{t}$.
\end{proof}

Subadditivity is a natural generalisation of quasi-homogeneity, as the
following simple observation states.

\begin{proposition}
  \label{proposition-KIP-ver-18:1}
  Each quasi-homogeneous function $\omega:[0,+\infty)\to [0,+\infty)$
  is subadditive.
\end{proposition}

\begin{proof}
  If $\omega$ is quasi-homogeneous and $s,t> 0$, following the proof
  of \cite[Theorem 7.2.4]{MR0025077},
  \Cref{proposition-Uniform-Cont-Ext:4} implies that
  \[
    \omega(s+ t) = s\cdot\mfrac{\omega\left(s+t\right)}{s+t} +
    t\cdot\mfrac{\omega\left(s+t\right)}{s+t}\leq
    s\cdot\mfrac{\omega\left(s\right)}{s} +
    t\cdot\mfrac{\omega\left(t\right)}{t} =
    \omega(s)+\omega(t).\qedhere
  \]
\end{proof}

Regarding \Cref{proposition-KIP-ver-18:1}, let us remark that the
function $\omega(t)=|\sin t|$, $t\geq 0$, is subadditive but not
quasi-homogeneous.\medskip

Suppose that $0<\alpha\leq 1$. Then the function
$\frac{t^\alpha}t=t^{\alpha-1}:(0,+\infty)\to (0,+\infty)$ is clearly
decreasing because
$\left(t^{\alpha-1}\right)'= (\alpha-1) t^{\alpha-2}\leq 0$. Hence, by
\Cref{proposition-Uniform-Cont-Ext:4,proposition-KIP-ver-18:1}, the
function $t^\alpha:[0,+\infty)\to [0,+\infty)$ is subadditive, and we
thus obtain the following important property of the symmetric
$d^\alpha$.

\begin{proposition}
  \label{proposition-Uniform-Cont-Ext:5}
  For a metric space $(X,d)$ and\/ $0<\alpha\leq 1$, the function
  $d^\alpha:X\times X\to \R$, defined as in \cref{eq:KIP-ver-09:1}, is
  also a metric on $X$.
\end{proposition}

The resulting metric space $\left(X,d^\alpha\right)$ is called a
\emph{H\"older metric space} of order $\alpha$. Evidently, in this
case, the $\alpha$-H\"older continuous maps from $(X,d)$ to a metric
space $(Y,\rho)$ are precisely the Lipschitz map from $(X,d^\alpha)$
to $(Y,\rho)$.\medskip 

An immediate consequence of \Cref{proposition-Uniform-Cont-Ext:5} is
that \Cref{theorem-Lipschitz-ext:12} actually encompasses the
following result, obtained by McShane in \cite[Corollary
1]{MR1562984}.

\begin{corollary}
  \label{corollary-KIP-ver-07:1}
  Let $(X,d)$ be a metric space, $A\subseteq X$ and $0<\alpha\leq
  1$. Then each $\alpha$-H\"older continuous function
  $\varphi:A\to \R$ with a H\"older coefficient $K\geq 0$ can be
  extended to an $\alpha$-H\"older continuous function $f:X\to \R$
  with the same H\"older coefficient.
\end{corollary}

In this section, we will show that \Cref{corollary-KIP-ver-07:1}
remains valid for $\alpha$-H\"older continuous maps between Hilbert
spaces, which is a result obtained by Minty in \cite[Theorem
1]{zbMATH03304171}.  In fact, we will show that this result is a
consequence of the considerations in the previous section and the
following inequality attributed to C. M. Fox, see
\cite{zbMATH03304171}.

\begin{lemma}
  \label{lemma-Kirszbraun-Intersc-v3:1}
  If\/ $\mathbf{u}_1,\dots, \mathbf{u}_k\in \R^m$ and\,
  $\lambda_1,\dots, \lambda_k\geq 0$, then
  \begin{equation}
    \label{eq:Kirszbraun-Intersc-v3:1}
    \sum_{i,j=1}^k
    \lambda_i\lambda_j\|\mathbf{u}_i-\mathbf{u}_j\|^{2\alpha} \leq
    \sum_{i,j=1}^k \lambda_i\lambda_j\left(\|\mathbf{u}_i\|^2+
      \|\mathbf{u}_j\|^2\right)^\alpha\quad
    \text{for every\, $0<\alpha\leq 1$.} 
  \end{equation}
\end{lemma}

\begin{proof}
  For $\alpha=1$, as in the proof of \Cref{theorem-Lip-Ext-Embb:1},
  the inequality holds for any choice of the coefficients
  $\lambda_1, \dots, \lambda_k$. For $\alpha < 1$, J. Moser and
  G. J. Minty \cite{zbMATH03304171} gave a simple proof of
  \cref{eq:Kirszbraun-Intersc-v3:1} by using the fact that the gamma
  function $\Gamma(\beta)$ can be represented in the form
  $\frac{\Gamma(\beta)}{(a+b)^\beta}=\medint\int_{0}^\infty
  e^{-at}e^{-bt}t^{\beta-1}d\!t$ for $a,b>0$ and $\beta>0$.  Thus,
  taking vectors $\mathbf{x}_1,\dots, \mathbf{x}_k\in \R^m$ and
  $a_1,\dots, a_k>0$, we obtain that
  \begin{equation}
    \label{eq:Kirszbraun-Intersc-v3:2}
    \sum_{i,j=1}^k\mfrac{\langle\mathbf{x}_i,\mathbf{x}_j\rangle}
    {(a_i+a_j)^\beta} 
    =\mfrac1{\Gamma(\beta)}\int_0^\infty
    \left\|\sum_{i=1}^k e^{-a_it}\mathbf{x}_i\right\|^2
    t^{\beta-1}d\!t\geq 0.
  \end{equation}
  Moreover, for nonzero vectors
  $\mathbf{u}_1, \dots, \mathbf{u}_k \in \mathbb{R}^n$ and
  $\lambda_1,\dots, \lambda_k\geq 0$, we also have that
  \begin{equation}
    \label{eq:Kirszbraun-Intersc-v3:3}
    \sum_{i,j=1}^k
    \lambda_i\lambda_j\|\mathbf{u}_i-\mathbf{u}_j\|^{2\alpha}=\sum_{i,j=1}^k
    \lambda_i\lambda_j\left(\|\mathbf{u}_i\|^2+
      \|\mathbf{u}_j\|^2\right)^\alpha\left[1-
      \mfrac{2\langle\mathbf{u}_i,\mathbf{u}_j\rangle}{\|\mathbf{u}_i\|^2+
        \|\mathbf{u}_j\|^2} \right]^\alpha.
  \end{equation}
  Since by Bernoulli's inequality,
  $\left[1-
    \frac{2\langle\mathbf{u}_i,\mathbf{u}_j\rangle}{\|\mathbf{u}_i\|^2+
      \|\mathbf{u}_j\|^2} \right]^\alpha \leq
  1-\frac{2\alpha\langle\mathbf{u}_i,\mathbf{u}_j\rangle}{\|\mathbf{u}_i\|^2+
    \|\mathbf{u}_j\|^2}$, finally we get that 
  \begin{align*}
    \sum_{i,j=1}^k
    \lambda_i\lambda_j
    &\left(\|\mathbf{u}_i\|^2+
      \|\mathbf{u}_j\|^2\right)^\alpha\left[1-
      \mfrac{2\langle\mathbf{u}_i,\mathbf{u}_j\rangle}{\|\mathbf{u}_i\|^2+
      \|\mathbf{u}_j\|^2} \right]^\alpha\\[-3pt]
    &\leq \sum_{i,j=1}^k
      \lambda_i\lambda_j
      \left(\|\mathbf{u}_i\|^2+
      \|\mathbf{u}_j\|^2\right)^\alpha
      -2\alpha
      \sum_{i,j=1}^k\mfrac{\langle\lambda_i\mathbf{u}_i,\lambda_j
      \mathbf{u}_j\rangle}{\left(\|\mathbf{u}_i\|^2+ 
      \|\mathbf{u}_j\|^2\right)^{1-\alpha}}.
  \end{align*}
  According to
  \cref{eq:Kirszbraun-Intersc-v3:2,eq:Kirszbraun-Intersc-v3:3}, this
  shows the inequality in \cref{eq:Kirszbraun-Intersc-v3:1}.
\end{proof}

Since for $<\alpha\leq 1$, the function $t^\alpha:[0,+\infty)\to \R$
is subadditive, we have the following immediate consequence of
\Cref{lemma-Kirszbraun-Intersc-v3:1}.

\begin{corollary}
  \label{corollary-Kirszbraun-Intersc-v3:1}
  If\/ $\mathbf{u}_1,\dots, \mathbf{u}_k\in \R^m$ and\,
  $\lambda_1,\dots, \lambda_k\geq 0$, then
  \begin{equation}
    \label{eq:Kirszbraun-Intersc-v3:4}
    \sum_{i,j=1}^k
    \lambda_i\lambda_j\|\mathbf{u}_i-\mathbf{u}_j\|^{2\alpha} \leq
    \sum_{i,j=1}^k 
    \lambda_i\lambda_j\left(\|\mathbf{u}_i\|^{2\alpha}+
      \|\mathbf{u}_j\|^{2\alpha}\right)\quad\text{for every\, $0<\alpha\leq 1$.} 
  \end{equation}
\end{corollary}

The above inequality significantly simplifies the proof of the
following extension theorem for $\alpha$-H\"older continuous maps,
established by Minty in \cite[Theorem 1]{zbMATH03304171}.

\begin{theorem}
  \label{theorem-Kirszbraun-Intersc-v4:2}
  Let $H$ be a Hilbert space, $(X,d)$ be a metric space,
  $A\subseteq X$ and\/ $0<\alpha\leq 1$. Then an $\alpha$-nonexpansive
  map $\varphi:A\to H$ can be extended to an $\alpha$-nonexpansive map
  $f:X\to H$ if either $\alpha\leq\frac 12$ or $X$ is an inner product
  space and $d$ is the metric generated by the inner product.
\end{theorem}

\begin{proof}
  Let us first assume that $\alpha\leq 1/2$. Next, let
  $d^\alpha,d^{2\alpha}:X\times X\to \R$ be defined as in
  \cref{eq:KIP-ver-09:1}. Since $<\alpha<2\alpha\leq 1$, by
  \Cref{proposition-Uniform-Cont-Ext:5}, both $\rho=d^\alpha$ and
  $\rho^2=d^{2\alpha}$ are metrics on $X$. Accordingly, it follows
  from \Cref{corollary-KIP-ver-17:1} (see also
  \Cref{corollary-KIP-ver-24:2}) that each $\rho$-nonexpansive map
  $\varphi:A\to H$ can be extended to a $\rho$-nonexpansive map
  ${f:X\to H}$. Since $\rho=d^\alpha$, by
  \Cref{proposition-KIP-draft-04:1}, this is equivalent to the fact
  that each $\alpha$-nonexpansive map $\varphi:A\to H$ can be extended
  to an $\alpha$-nonexpansive map $f:X\to H$.\smallskip

  Suppose now that $X$ is an inner product space, $d$ is the metric
  generated by the inner product, and there is no restriction on
  $0<\alpha\leq 1$.  Then according to
  \Cref{proposition-Uniform-Cont-Ext:5} and the inequality in
  \cref{eq:Kirszbraun-Intersc-v3:4}, the metric $\rho=d^\alpha$ on $X$
  satisfies the Gr\"unbaum-Zarantonello condition in
  \cref{eq:KIP-ver-17:1}. Therefore, \Cref{theorem-KIP-ver-17:1} now
  implies that any $\rho$-nonexpansive map $\varphi: A \to H$ can be
  extended to a $\rho$-nonexpansive map ${f: X \to H}$ or, in other
  words, that any $\alpha$-nonexpansive map $\varphi: A \to H$ can be
  extended to an $\alpha$-nonexpansive map $f: X \to H$.
\end{proof}

Let us remark that the proof of \Cref{theorem-Kirszbraun-Intersc-v4:2}
given in \cite{zbMATH03304171} is based on
``$K$-functions''. Moreover, the case in this theorem when $(X,d)$ is
an inner product is a direct consequence of the Kirszbraun-Valentine
extension theorem (\Cref{theorem-Lipschitz-ext:6}) and some results on
isometric embeddings in Hilbert spaces due to John von Neumann, see
\cite[Theorem 3]{Schoenberg1937} and \cite[Corollary
1]{Schoenberg1938}.\medskip

We conclude this section with some remarks regarding the role of the
condition that $\alpha\leq 1$. The following simple observation is
well known.

\begin{proposition}
  \label{proposition-Lipschitz-ext:6}
  Let $\Delta\subseteq \R$ be an interval and $\alpha>1$. Then each
  $\alpha$-nonexpansive function $f:\Delta\to \R$\/ is constant.
\end{proposition}

\begin{proof}
  Suppose that $f:\Delta\to \R$ is $\alpha$-nonexpansive and take an
  interior point $s\in \Delta$. Then for $t\in \Delta$ with $t\neq s$,
  it follows that
  $0\leq \frac{|f(t)-f(s)|}{|t-s|}\leq
  |t-s|^{\alpha-1}\xrightarrow[t\to s]{} 0$ because $\alpha>1$. Hence,
  $f'(s)=0$ which implies that $f$ is a constant function being
  continuous on the interval $\Delta$.
\end{proof}

Based on this, one can easily construct $\alpha$-H\"older continuous
functions for $\alpha > 1$ that cannot be extended to
$\alpha$-H\"older continuous functions.

\begin{example}
  \label{example-Lipschitz-ext:1}
  Take any $\alpha>1$ and define a function $\varphi:\{0,2\}\to \R$ by
  $\varphi(0)=2^{\alpha-1}$ and $\varphi(2)=2^\alpha$. Then $\varphi$
  is $\alpha$-nonexpansive because
  \[
    |\varphi(2)-\varphi(0)|= \left|2^\alpha-2^{\alpha-1}\right|=
    2^{\alpha-1}< 2^\alpha=|2-0|^\alpha.
  \]
  However, $\alpha>1$ and $\varphi$ is not constant. Therefore, by
  \Cref{proposition-Lipschitz-ext:6}, $\varphi$ cannot be extended to
  an $\alpha$-nonexpansive function $f:\R\to \R$.\qed
\end{example}

\section{Concave Functions and Affine Majorants}
\label{sec:conc-funct-affine}

A function $\omega:\Delta\to \R$, defined on an interval
$\Delta\subseteq \R$, is \emph{concave} if the line segment between
any two points on the graph of $\omega$ lies below or on the graph of
$\omega$. In other words, $\omega$ is concave when for every
$s,t\in \Delta$ and $\lambda\in [0,1]$,
\[
  \omega\big((1-\lambda)s+\lambda t\big)\geq (1-\lambda) \omega(s) +\lambda
  \omega(t).
\]
Concave functions are also called \emph{concave downwards} or
\emph{convex upwards}.\medskip

Concave functions are naturally related to both quasi-homogeneous
functions and subadditive ones. Namely, according to the defining
property in \cref{eq:Uniform-Cont-Ext:2}, a function
$\omega:[0,+\infty)\to [0,+\infty)$ is quasi-homogeneous if
$\omega(\lambda t)\leq \lambda\omega(t)$ for every $\lambda\geq 1$ and
$t\geq 0$. Evidently, for $0<\lambda\leq1$, this is equivalent to the
property that
$\omega(t)=\omega\left(\mfrac \lambda\lambda t\right)\leq
\mfrac1\lambda \omega(\lambda t)$, i.e.\ that
$\lambda\omega(t)\leq \omega(\lambda t)$. Based on this, we have that
each concave function is quasi-homogeneous. For the special case that
each concave function is subadditive, see e.g.\ \cite[Theorem
6.2.4]{MR0025077}.

\begin{proposition}
  \label{proposition-Lipschitz-ext:4}
  Each concave function $\omega:[0,+\infty)\to [0,+\infty)$ is
  quasi-homogeneous. In particular, each concave function
  $\omega:[0,+\infty)\to [0,+\infty)$ is also subadditive.
\end{proposition}

\begin{proof}
  If $\omega$ is concave, $t\geq 0$ and $\lambda\in [0,1]$, then
  \[
    \omega(\lambda t)= \omega\big((1-\lambda)0+\lambda t\big)\geq
    (1-\lambda) \omega(0)+\lambda \omega(t)\geq
    \lambda\omega(t).
  \]
  Hence, $\omega$ is quasi-homogeneous and, by
  \Cref{proposition-KIP-ver-18:1}, it is also subadditive.
\end{proof}

Regarding \Cref{proposition-Lipschitz-ext:4}, let us remark the
function $\omega(0)=0$ and $\omega(t)=\frac1t$, for $t>0$, is clearly
quasi-homogeneous but not concave.\medskip

Another interesting property of concave functions is that they are
increasing. This was remarked in the proof of \cite[Theorem
2]{MR1562984} and shown in \cite[Proposition 3]{zbMATH07212343}.

\begin{proposition}
  \label{proposition-Uniform-Cont-Ext:1}
  For every $r\geq 0$, each concave function
  $\omega:[r,+\infty)\to [0,+\infty)$ is increasing.
\end{proposition}

\begin{proof}
  Take points $r\leq s<t$ and set $t_n=s+n(t-s)$, $n\in \N$. Then
  $t=\left(1-\frac1n\right) s+ \frac1nt_n$ for every $n\in \N$. Since
  $\omega$ is concave, this implies that
  $\omega(t)\geq \left(1-\frac1n\right) \omega(s)+
  \frac1n\omega(t_n)$, $n\in \N$. Accordingly,
  $\omega(s)\leq \omega(t)$ because
  \[
    \omega(t)-\omega(s)\geq -\mfrac1n \omega(s)+ \mfrac1n\omega(t_n)
    \geq -\mfrac1n \omega(s)\xrightarrow[n\to\infty]{}0.\qedhere
  \]
\end{proof}

It was shown in \cite[Theorem 11.A]{Roberts1973} that a concave
function $\omega:[0,+\infty)\to [0,+\infty)$ is not only continuous on
$(0,+\infty)$, but also Lipschitz on each closed bounded interval
$\Delta\subseteq (0,+\infty)$. In the next proposition, we slightly
generalise this fact.

\begin{proposition}
  \label{proposition-Rem-Lipschitz:2}
  If $\omega:(0,+\infty)\to [0,+\infty)$ is a concave function, then
  $\omega\uhr [a,+\infty)$ is Lipschitz for each $a>0$. Consequently,
  $\omega$ is also continuous.
\end{proposition}

\begin{proof}
  Take $a> b>0$ and consider the auxiliary function
  $\eta:[0,+\infty)\to [0,+\infty)$ defined by $\eta(s)=\omega(b+s)$
  for every $s\geq0$. It is trivial to see that $\eta$ is also concave
  because it has the same values as $\omega\uhr [b,+\infty)$. Hence,
  according to
  \Cref{proposition-Lipschitz-ext:4,proposition-Uniform-Cont-Ext:1},
  the function $\eta$ is both quasi-homogeneous and increasing. Then
  for points ${a-b\leq s<r<t}$, these properties of $\eta$ imply that
  $ \frac s{r}\eta(r)\leq \eta(s)\leq \eta(t)\leq \frac t{r}\eta(r)$.
  Therefore, by \Cref{proposition-Uniform-Cont-Ext:4}, this also
  implies that
  $0\leq \eta(t)-\eta(s)\leq \frac{\eta(r)}{r} (t-s)\leq
  \frac{\eta(a-b)}{a-b} (t-s)$. Thus, $\eta\uhr [a-b,+\infty)$ is
  Lipschitz and so is $\omega\uhr[a,+\infty)$ because
  $\omega(t)=\eta(t-b)$ and the translation $t-b$, $t\geq 0$, is an
  isometry.
\end{proof}

\Cref{proposition-Rem-Lipschitz:2} implies the following
characteristic property of continuity of concave functions.

\begin{corollary}
  \label{corollary-KIP-ver-18:1}
  A concave function $\omega:[0,+\infty)\to [0,+\infty)$ is continuous
  if and only if $\omega(0)= \inf_{s>0}\omega(s)$.
\end{corollary}

\begin{proof}
  The set $\Delta=\omega\left((0,+\infty)\right)\subseteq \R$ is an
  interval because, according to \Cref{proposition-Rem-Lipschitz:2},
  $\omega\uhr (0,+\infty)$ is continuous. Moreover, by
  \Cref{proposition-Uniform-Cont-Ext:1}, $\omega$ is
  increasing. Hence, by an elementary fact of real analysis, the
  increasing function $\omega$ is continuous precisely when
  $\Delta\cup\{\omega(0)\}\subseteq \R$ is an interval, i.e.\ when
  $\omega(0)= \inf_{s>0}\omega(s)$.
\end{proof}

A quasi-homogeneous function $\omega:[0,+\infty)\to [0,+\infty)$ does
not depend on its value at ${t=0}$ because
${0\leq \omega(\lambda 0)=\omega(0)\leq \lambda\omega(0)}$ for every
$\lambda\geq 1$.  The same is valid for each subadditive function
$\omega:[0,+\infty)\to [0,+\infty)$ because
${\omega(0+t)=\omega(t)\leq \omega(0)+\omega(t)}$ for every
$t>0$. Thus, a function $\omega:[0,+\infty)\to [0,+\infty)$ is
quasi-homogeneous or subadditive precisely when so is the
restriction $\omega\uhr (0,+\infty)$.  Concave functions have a
similar property.

\begin{proposition}
  \label{proposition-Uniform-Cont-Ext:8}
  Let $\omega:[0,+\infty)\to [0,+\infty)$ be a function such that
  $\omega\uhr (0,+\infty)$ is concave. Then $\omega$ is concave if and
  only if\/ $\omega(0)\leq \inf_{s>0}\omega(s)$.
\end{proposition}

\begin{proof}
  Assume that $\omega$ is concave. Then by Proposition
  \ref{proposition-Uniform-Cont-Ext:1}, $\omega$ is increasing and
  therefore $\omega(0)\leq \omega(s)$ for every $s\geq0$.  Conversely,
  assume that $\omega(0)\leq \inf_{s>0} \omega(s)$ and take $t>0$ and
  $0<\lambda<1$. Then for every $0<s<t$, it follows that
  \[
    \omega\left((1-\lambda)s+\lambda t\right)\geq
    (1-\lambda)\omega(s)+ \lambda\omega(t)\geq
    (1-\lambda)\omega(0)+\lambda \omega(t).
  \]
  However, by \Cref{proposition-Rem-Lipschitz:2},
  $\omega\uhr (0,+\infty)$ is continuous. Therefore,
  \[
    \omega\left((1-\lambda)0+\lambda t\right) =\lim_{s\to 0}
    \omega\left((1-\lambda)s+\lambda t\right)\geq
    (1-\lambda)\omega(0)+\lambda \omega(t).\qedhere
  \]
\end{proof}

Another interpretation of the characteristic property in
\Cref{proposition-Uniform-Cont-Ext:8} is that each concave function
${\omega:(0,+\infty)\to [0,+\infty)}$ can be extended to a concave
function on the interval ${[0,+\infty)}$. Based on this
interpretation, we will study properties of concave functions
${\omega:(0,+\infty)\to [0,+\infty)}$ by using corresponding
properties of the concave functions
$\omega:[0,+\infty)\to [0,+\infty)$.\medskip

A special case of concave functions is given by the so called affine
functions. Let us recall that a function $\alpha:\R\to \R$ is
\emph{affine} if
\[
  \alpha((1-\lambda)s +\lambda t)=(1-\lambda)\alpha(s)+\lambda
  \alpha(t)\quad \text{for every $s,t\in \R$ and $0\leq \lambda\leq 1$.} 
\]
It is a simple exercise that a function $\alpha:\R\to \R$ is affine
precisely when it is changing linearly, i.e.\ when $\alpha(t)=at+b$
for some constants $a,b\in \R$. The relation between affine function
and concave functions is simple. The following theorem is implicitly
present in \cite[Theorem 12.D]{Roberts1973}.

\begin{theorem}
  \label{theorem-Rem-Lipschitz:1}
  A function $\omega:(0,+\infty)\to [0,+\infty)$ is concave if and
  only if there exists a family $\mathcal{A}$ of affine functions
  $\alpha:\R\to \R$ such that
  \begin{equation}
    \label{eq:Rem-Lipschitz:1}
    \omega(t)=\min_{\alpha\in \mathcal{A}}\alpha(t)\quad \text{for 
      every\/ $t> 0$.}   
  \end{equation}
\end{theorem}

\begin{proof}
  Let $\mathcal{A}$ be a family of affine functions satisfying the
  condition in \cref{eq:Rem-Lipschitz:1}. Then each function
  $\alpha\in \mathcal{A}$ is also concave. Hence, by
  \cref{eq:Rem-Lipschitz:1}, so is the function $\omega(t)$ because
  for $s,t> 0$ and $\lambda\in [0,1]$, we get that
  \begin{align*}
    (1-\lambda) \omega(s) +\lambda \omega(t)
    &= (1-\lambda)
      \min_{\alpha\in \mathcal{A}}
      \alpha(s) +\lambda
      \min_{\alpha\in \mathcal{A}}
      \alpha(t)\\
    &\leq \inf_{\alpha\in \mathcal{A}}\left[ (1-\lambda)
      \alpha(s) +\lambda
      \alpha(t)\right]\\
    &= \inf_{\alpha\in \mathcal{A}} \alpha\left((1-\lambda) s +\lambda
      t\right)
      = \omega\left((1-\lambda) s +\lambda
      t\right).  
  \end{align*}

  Conversely, assume that $\omega:(0,+\infty)\to [0,+\infty)$ is a
  concave function. The essential element in the proof is now the fact
  that $\omega(t)$ has one-sided derivatives at each point
  $t\in (0,+\infty)$. In the setting of convex functions, this was
  actually obtained in \cite[Theorem 11.B]{Roberts1973}. We follow the
  arguments there, slightly clarifying the proof of this fact. To this
  end, for points $0<x<y$, let
  $\mu_{xy}=\frac{\omega(y)-\omega(x)}{y-x}$ be the slope of the
  straight line passing through the points $(x,\omega(x))$ and
  $(y,\omega(y))$ of the graph of the function $\omega(t)$. Then,
  since $\omega(t)$ is concave, it follows that
  \begin{equation}
    \label{eq:KIP-ver-12:1}
    \mu_{xy}\geq \mu_{xz}\geq \mu_{yz}\quad \text{whenever $0<x<y<z$.}
  \end{equation}

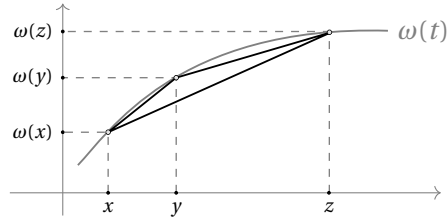
  \begin{center}
    \begin{tikzpicture}
      \draw[white!50!black,->] (-.7,0) -- (5,0);
      \draw[white!50!black,->] (0,-.3) -- (0,2.5);
      \draw[gray,dashed,line width=0.5pt,-] (.6,0) -- (.6,.8) -- (0,.8);
      \draw[gray,dashed,line width=0.5pt,-] (3.525,0) -- (3.525,2.13)
      -- (0,2.13);
      \draw[gray,dashed,line width=0.5pt,-]
      (1.5,0) -- (1.5,1.52) -- (0,1.52);
      \draw[black!50!white,line width=.7pt] (.2,.365) .. controls
      (.6,.78) 
      and (1.5,2.22) .. (4.3,2.14) (4.3,2.14) node[right] {\small $\omega(t)$};
      \filldraw [black] (3.527,2.12) circle [radius=.8pt] (0,2.13)
      circle [radius=.5pt] node[left] {\tiny $\omega(z)$};
      \filldraw [black] (.6,.801) circle [radius=.8pt] (0,.8) circle
      [radius=.5pt] node[left] {\tiny $\omega(x)$};
      \filldraw [black] (1.5,1.52) circle [radius=.8pt] (0,1.52)
      circle [radius=.5pt] node[left] {\tiny $\omega(y)$};
      \filldraw [black] (.6,0) circle [black, radius=.5pt] node[below]
      {\tiny $x$} (1.5,0) circle [black, radius=.5pt] node[below]
      {\tiny $y$} (3.52,0) circle [black, radius=.5pt] node[below]
      {\tiny $z$};
      \draw[black,line width=0.7pt,-] (.6,.8) -- (3.55,2.12) --
      (1.5,1.52) -- (.6,.8);
      \filldraw [white!90!black] (3.527,2.12) circle [radius=.5pt]
      (.6,.801) circle [radius=.5pt] (1.5,1.52) circle [radius=.5pt];
    \end{tikzpicture}
  \end{center}
  This easily follows from the fact that the intermediate point $y$ is
  an affine combination of the points $x$ and $z$ with positive
  coefficients, namely, that
  $y=\frac{z-y}{z-x}\cdot x+\frac{y-x}{z-x}\cdot z$. Thus taking
  $\lambda=\frac{y-x}{z-x}>0$ and $1-\lambda=\frac{z-y}{z-x}>0$, the
  concavity of $\omega(t)$ implies that
  \begin{align*}
    \mu_{xy}&=\frac{\omega\left((1-\lambda)x+\lambda z\right)-
              \omega(x)}{y-x} \\
            &\geq \frac{(1-\lambda)\omega(x)+\lambda \omega(z)-
              \omega(x)}{y-x} 
              = \lambda\cdot \frac{\omega(z)-\omega(x)}{y-x}=
              \frac{\omega(z)-\omega(x)}{z-x}=\mu_{xz}. 
  \end{align*}
  A similar argument shows that $\mu_{xz}\geq \mu_{yz}$ as well,
  therefore \cref{eq:KIP-ver-12:1} holds.\smallskip

  We can now proceed as in the proof of \cite[Theorem
  11.B]{Roberts1973}. For this purpose, fix a point $t>0$ and take
  other points $0<s<p<t<q<r$. Then from \cref{eq:KIP-ver-12:1} we
  obtain that $\mu_{st}\geq \mu_{pt}\geq \mu_{tq}\geq
  \mu_{tr}$. Accordingly, $\mu_{st}$, $s<t$, is a decreasing bounded
  function and there exists the limit
  $\omega'_-(t)=\lim_{s\to t^-}\mu_{st}$. In fact, $\omega_-'(t)$ is
  the \emph{left derivative} at the point $t$. For the same reason,
  the \emph{right derivative}
  $\omega'_+(t)= \lim_{r\to t^+}\mu_{tr}$ also exists, and have that
  \begin{equation}
    \label{eq:KIP-ver-12:2}
    \mu_{st}\geq \omega'_-(t)\geq \omega'_+(t)\geq \mu_{tr}\quad
    \text{whenever $0<s<t<r$.}
  \end{equation}
  Based on this, we can finalise the proof as in \cite[Theorem
  12.D]{Roberts1973}. Briefly, take any
  $\mu_t\in [\omega_+'(t),\omega_-'(t)]$ and consider the affine
  function defined by $\alpha_t(x)= \mu_t(x-t)+\omega(t)$. Then
  $\alpha_t(t)=\omega(t)$ and according to \cref{eq:KIP-ver-12:2}, for
  $0<s<t$ it follows that
  \[
    \omega(t)-\omega(s)=\mu_{st}\cdot(t-s)\geq
    \omega_-'(t)(t-s)\geq \mu_t(t-s).
  \]
  This shows that $\alpha_t(s)= \mu_t(s-t)+\omega(t)\geq
  \omega(s)$. Similarly, we also have that
  ${\alpha_t(r)\geq \omega(r)}$ for every $r>t$. Therefore, the family
  $\mathcal{A}=\left\{\alpha_t:t>0\right\}$ is as required in
  \cref{eq:Rem-Lipschitz:1}.
\end{proof}

Complementary to \Cref{theorem-Rem-Lipschitz:1} is the following
result, obtained by Aronszajn and Panitchpakdi in \cite[Theorem
1]{MR84762}. A special case of this result was also obtained by
McShane in \cite[Lemma]{MR1562984}.

\begin{theorem}
  \label{theorem-Uniform-Cont-Ext:1}
  An increasing function $\eta:[0,+\infty)\to [0,+\infty)$ admits a
  concave function $\omega:[0,+\infty)\to [0,+\infty)$ with\/
  $\eta\leq \omega$ and\,
  $\lim_{t\to 0}\omega(t)=\lim_{t\to 0}\eta(t)$, if and only if
  \begin{equation}
    \label{eq:Uniform-Cont-Ext:1}
    \varlimsup_{t\to
      +\infty}\mfrac{\eta(t)}t<+\infty.
  \end{equation}
\end{theorem}

Each concave function ${\omega:[0,+\infty)\to [0,+\infty)}$ satisfies
\cref{eq:Uniform-Cont-Ext:1}, since, according to
\Cref{proposition-Lipschitz-ext:4}, it is quasi-homogeneous and
therefore
\[
  \varlimsup_{t\to +\infty} \mfrac{\omega(t)}t\leq \sup_{t\geq
    1}\mfrac{\omega(t)}t=\sup_{t\geq 1}\mfrac{\omega(t\cdot1)}t\leq
  \omega(1)<+\infty.
\]
Thus, the proof of \Cref{theorem-Uniform-Cont-Ext:1} is reduced to the
construction of the corresponding concave majorant
$\omega:[0,+\infty)\to [0,+\infty)$. This construction is essentially
the same as that in the proof of \cite[Theorem 1]{MR84762}, and is
based on the following property.

\begin{proposition}
  \label{proposition-Lipschitz-ext:9}
  If\/ $\eta:[0,+\infty)\to [0,+\infty)$\/ is an increasing function
  which satisfies \cref{eq:Uniform-Cont-Ext:1}, then there are\/
  $\tau,M>0$ such that
  \begin{equation}
    \label{eq:Lipschitz-ext:26}
    \eta(t)\leq \eta(s)+ \mfrac{\tau M}s t\quad \text{for every\/
      $0<s<\tau$ and $t\geq 0$.} 
  \end{equation}
\end{proposition}

\begin{proof}
  Since
  $\varlimsup_{t\to
    +\infty}\frac{\eta(t)}t=\lim_{t\to+\infty}\left(\sup_{s\geq t}
    \frac{\eta(s)}s\right)<+\infty$, there are positive numbers
  $\tau, M>0$ such that $\eta(t)\leq Mt$ for every $t\geq \tau$. To
  see that these $\tau$ and $M$ are as in \cref{eq:Lipschitz-ext:26},
  take $0<s<\tau$ and $t\geq 0$.  If $t\leq s$, then
  $\eta(t)\leq \eta(s)\leq \eta(s)+ \frac{\tau M}s t$ because $\eta$
  is increasing. Similarly, if $t>\tau$, then
  $\eta(t)\leq M t\leq \frac\tau s Mt\leq \eta(s)+ \frac{\tau M}s t$
  because $\eta(s)\geq 0$ and $\frac\tau s>1$. Finally, for
  $s < t\leq \tau$, we have that $\frac ts > 1$ and for the same
  reason as before we obtain that
  $\eta(t)\leq \eta(\tau)\leq M\tau < \frac\tau s Mt \leq \eta(s) +
  \frac{\tau M}s t$. Thus, \cref{eq:Lipschitz-ext:26} holds.
\end{proof}

The following application of \Cref{proposition-Lipschitz-ext:9}
completes the proof of \Cref{theorem-Uniform-Cont-Ext:1}.

\begin{proposition}
  \label{proposition-Uniform-Cont-Ext:6}
  If\/ $\eta:[0,+\infty)\to [0,+\infty)$\/ is an increasing function
  which satisfies \cref{eq:Uniform-Cont-Ext:1}, then there is a family
  $\mathcal{A}$ of affine functions $\alpha:\R\to \R$ such that
  \begin{equation}
    \label{eq:Uniform-Cont-Ext:6}
    \eta(t)\leq \inf_{\alpha\in \mathcal{A}}\alpha(t)\ \ \text{for
      every $t\geq0$,\quad and}\quad \lim_{t\to
      0^+}\inf_{\alpha\in \mathcal{A}}\alpha(t)=\lim_{t\to 0}\eta(t).  
  \end{equation}
  In particular, the function $\omega:[0,+\infty)\to [0,+\infty)$
  defined by $\omega(t)=\inf_{\alpha\in\mathcal{A}}\alpha(t)$,
  $t\geq 0$, is a concave majorant of\/ $\eta$ such that\/
  $\lim_{t\to 0}\omega(t)=\lim_{t\to 0}\eta(t)$.
\end{proposition}

\begin{proof}
  Let ${\tau,M>0}$ be as \cref{eq:Lipschitz-ext:26}, applied to the
  given the function $\eta$. Next, for each $0<s<\tau$, define an
  affine function $\alpha_s:\R\to \R$ by
  \begin{equation}
    \label{eq:Uniform-Cont-Ext:8}
    \alpha_s(t)=\eta(s)+ \mfrac{\tau M}s t\quad \text{for every $t\in
      \R$.} 
  \end{equation}
  Then by \cref{eq:Lipschitz-ext:26}, the family
  $\mathcal{A}=\left\{\alpha_s: 0<s<\tau\right\}$ satisfies the
  inequality in \cref{eq:Uniform-Cont-Ext:6}, so we can define
  $\omega:[0,+\infty)\to [0,+\infty)$ by
  $\omega(t)= \inf_{0<s<\tau}\alpha_s(t)$ for every $t\geq 0$. In this
  setting, if $0<t<\tau^2$, then $s=\sqrt{t}<\tau$ and it follows from
  \cref{eq:Uniform-Cont-Ext:8} that
  \[
    \lim_{t\to 0}\eta(t)\xleftarrow[0\gets t]{}
    \eta(t)\leq\omega(t)\leq \alpha_{\sqrt{t}}(t) =
    \eta\left(\sqrt{t}\right) +\tau M\sqrt{t}\xrightarrow[t\to 0]{}
    \lim_{t\to 0}\eta(t).
  \]
  Thus, $\lim_{t\to 0} \omega(t)=\lim_{t\to 0}\eta(t)$. Finally, since
  each function $\alpha_s(t)$, $0<s<\tau$, is affine, as in the proof
  of \Cref{theorem-Rem-Lipschitz:1}, the function $\omega(t)$ is
  concave. Indeed, for $x,y\geq 0$ and $\lambda\in [0,1]$, we get that
  \begin{align*}
    (1-\lambda) \omega(x) +\lambda \omega(y)
    &= (1-\lambda)
      \inf_{0<s<\tau}
      \alpha_s(x) +\lambda
      \inf_{0<s<\tau}
      \alpha_s(y)\\
    &\leq \inf_{0<s<\tau}\big[ (1-\lambda)
      \alpha_s(x) +\lambda
      \alpha_s(y)\big]\\
    &= \inf_{0<s<\tau} \alpha_s\big((1-\lambda) x +\lambda
      y\big)\\
    &= \omega\big((1-\lambda) x +\lambda
      y\big). \qedhere
  \end{align*}
\end{proof}

\section{Gr\"unbaum-Zarantonello's Extension Theorem}
\label{sec:grunb-zarant-extens}

The \emph{least modulus of continuity} of a map $f:X\to Y$ between
metric spaces $(X,d)$ and $(Y,\rho)$ is the function
$\omega_f:[0,+\infty)\to [0,+\infty]$ defined by
\begin{equation}
  \label{eq:Lipschitz-ext:16}
  \omega_f (t) = \sup_{d(p,q)\leq t} \rho (f(p), f(q)),\quad t\geq 0.
\end{equation}
In general, $\omega_f$ is an increasing extended function with
$\omega_f(0)=0$. Moreover, it is the smallest function
$\omega:[0,+\infty)\to [0,+\infty]$ with the property that
\begin{equation}
  \label{eq:Lipschitz-ext:17}
  \rho(f(p),f(q))\leq \omega(d(p,q))\quad \text{for every $p,q\in X$.}
\end{equation}
Each \emph{increasing function} $\omega:[0,+\infty)\to [0,+\infty]$
which satisfies \cref{eq:Lipschitz-ext:17} is termed a \emph{modulus
  of continuity for $f$}.  Thus, the function $\omega_f$ is the
\emph{least modulus of continuity} that $f$ obeys. Each modulus of
continuity for a map $f:X\to Y$ still represents some measure of how
fast this function can change. For instance, $f:X\to Y$ is
$K$-Lipschitz, $K\geq 0$, precisely when $\omega(t)=Kt$, $t\geq 0$, is
a modulus of continuity for $f$, i.e.\ when $f$ has a \emph{linear}
modulus of continuity. Similarly, $f:X\to Y$ is $\alpha$-H\"older
continuous with a coefficient $K\geq0$ if and only if
$\omega(t)=K t^\alpha$, $t\geq 0$, is a modulus of continuity for
$f$.\medskip

In the presence of subadditivity, each modulus of continuity does not
take infinite values and is either a constant function or is $0$ only
at $t=0$.

\begin{proposition}
  \label{proposition-Uniform-Cont-Ext:3}
  Let $\omega:[0,+\infty)\to [0,+\infty]$ be an extended increasing
  subadditive function. If\/ $\omega(s)<+\infty$ \textup{(}in
  particular, $\omega(s)=0$\textup{)} for some $s>0$, then
  $\omega(t)<+\infty$ \textup{(}respectively, $\omega(t)=0$\textup{)}
  for every $t\geq0$.
\end{proposition}

\begin{proof}
  Since $\omega$ is subadditive, it follows that
  $\omega(ns)\leq n\omega(s)$ for every $n\in\N$. Hence, the property
  is evident from the fact that $\omega$ is increasing.
\end{proof}

It is well known that a map $f:X\to Y$ between metric spaces is
\emph{uniformly continuous} precisely when it possesses a modulus of
continuity $\omega$ with the property that $\lim_{t\to
  0}\omega(t)=0$. In fact, in many sources, the condition
$\lim_{t\to 0}\omega(t)=0$ is explicitly included in the definition of
a modulus of continuity, and such a modulus is sometimes referred to
as a \emph{modulus of uniform continuity} (see, e.g.,
\cite{MR84762}). In the rest of this paper, we will adopt this
convention to emphasise the condition $\lim_{t\to 0}\omega(t)=0$. In
this context, the significance of subadditivity becomes evident as a
key element in the following extension theorem for uniformly
continuous functions, obtained by McShane in \cite[Theorem
2]{MR1562984}, see also Valentine \cite{zbMATH03102006}. 

\begin{theorem}
  \label{theorem-Lipschitz-ext:8}
  If\/ $(X,d)$ is a metric space, $A\subseteq X$ and $\varphi:A\to \R$
  has a subadditive modulus of uniform continuity
  $\omega:[0,+\infty)\to [0,+\infty]$, then $\varphi$ can be extended
  to a function $f:X\to \R$ preserving the same modulus of continuity
  $\omega$.
\end{theorem}

\begin{proof}
  Since $\lim_{t\to 0}\omega(t)=0$, by
  \Cref{proposition-Uniform-Cont-Ext:3}, $\omega$ does not take
  infinite values. For the same reason, since the extension property
  is trivial for constant functions, we can assume that $\omega(t)>0$
  for every $t>0$.  Hence, we can define another metric $\rho$ on $X$
  by $\rho(p,q)=\omega(d(p,q))$ for every $p, q\in X$. Then by
  \cref{eq:Lipschitz-ext:17}, $\varphi$ is nonexpansive (i.e.\
  $1$-Lipschitz) with respect to $\rho$ because $\omega$ is a modulus
  of continuity for $\varphi$. Accordingly, by the McShane-Whithney
  Extension Theorem (\Cref{theorem-Lipschitz-ext:12}), $\varphi$ can
  be extended to a nonexpansive function $f:X\to \R$ on the metric
  space $(X,\rho)$. Finally, since
  ${|f(p)- f(q)|\leq \rho(p,q)=\omega(d(p,q))}$ for every $p,q\in X$,
  it follows from \cref{eq:Lipschitz-ext:17} that $\omega$ is also a
  modulus of continuity for $f$ with respect to the metric $d$ on $X$.
\end{proof}

In other words, \Cref{theorem-Lipschitz-ext:8} states that if
$\varphi:A\to \R$ is uniformly continuous and has a subadditive
modulus of uniform continuity, then it can be extended to a uniformly
continuous function preserving the same modulus of uniform
continuity. In the original proof of this theorem, McShane gave an
explicit construction for the extension. Namely, he constructed this
extension exactly as in \Cref{theorem-Lipschitz-ext:12}, but now using
the metric $\rho$ on $X$ instead of the original metric $d$. Another
notable difference is that McShane formulated his \cite[Theorem
2]{MR1562984} for a concave modulus of uniform continuity, rather than
a subadditive one.\medskip

In \cite[Corollary 2]{MR1562984}, based on
\Cref{theorem-Lipschitz-ext:8}, McShane obtained the following further
result regarding the extension problem of uniformly continuous
functions.

\begin{corollary}
  \label{corollary-KIP-ver-21:1}
  If\/ $(X,d)$ is a metric space, $A\subseteq X$ is a nonempty set and
  $\varphi:A\to \R$ is bounded and uniformly continuous, then it can
  be extended to the whole of\/ $X$ preserving the uniform continuity
  and the bounds.
\end{corollary}

\begin{proof}
  Let $a=\inf \varphi(A)$ and $b=\sup \varphi(A)$ so that
  $\varphi:A\to [a,b]$. Also, let $\eta=\omega_\varphi$ be the least
  modulus of continuity for $\varphi$. Then $\eta$ is a bounded
  increasing function with $\lim_{t\to 0}\eta(t)=0$. Accordingly,
  $\eta$ satisfies \cref{eq:Uniform-Cont-Ext:1} because
  $\lim_{t\to\infty} \frac{\eta(t)}t=0$. Hence, by
  \Cref{theorem-Uniform-Cont-Ext:1} (see also
  \Cref{proposition-Lipschitz-ext:4,proposition-Uniform-Cont-Ext:1}),
  $\varphi$ admits a subadditive modulus of continuity $\omega$ with
  $\lim_{t\to 0}\omega(t)=\lim_{t\to 0}\eta(t)=0$. Therefore, by
  \Cref{theorem-Lipschitz-ext:8}, $\varphi$ can be extended to a
  uniformly continuous function $\Phi:X\to \R$.  Finally, define
  $f:X\to [a,b]$ by $f(x)=a$ if $\Phi(x)<a$, $f(x)=\Phi(x)$ if
  $a\leq \Phi(x)\leq b$ and $f(x)=b$ if $\Phi(x)>b$. Then $f$ remains
  an extension of $\varphi$ with the same modulus of uniform
  continuity $\omega$.
\end{proof}

In 1990, Mandelkern \cite{MR1076066} presented an alternative proof of
\Cref{corollary-KIP-ver-21:1} by showing that a construction for
continuous extensions\,---\,similar to the one given by Tietze in
\cite{Tietze1914a} but proposed by Dieudonn\'e in
\cite{MR0120319}\,---\,actually preserves uniform continuity as well.
Namely, he proved the following theorem.

\begin{theorem}
  \label{theorem-Mandelkern:1}
  For a metric space\/ $(X,d)$, a nonempty closed set $A\subseteq X$
  and $\varphi:A\to [1,2]$, define a function $f:X\to \R$ by
  $f\uhr A=\varphi$ and
  \[
    f(p)= \inf_{a\in A}\varphi(a)\cdot\mfrac{d(p,a)}{d(p,A)}\quad
    \text{for every $p\in X\setminus A$.}
  \]
  Then $f$ is uniformly continuous whenever so is $\varphi$.
\end{theorem}

In fact, as shown in \cite{Gutev2022}, all known Tietze-type
constructions for continuous extensions preserve uniform
continuity. Finally, going back to McShane's result in
\Cref{corollary-KIP-ver-21:1}, let us also remark that recently
G. Beer \cite{zbMATH07212343} showed that a bounded uniformly
continuous function on a subset $A\subseteq X$ is actually Lipschitz
with respect to a metric on $X$ that is uniformly equivalent to
$d$.\medskip

The extension property in \Cref{corollary-KIP-ver-21:1} does not hold
for unbounded functions ${\varphi:A\to \R}$, a fact naturally linked
to convexity in metric spaces. A metric space $(X,d)$ is called
\emph{Menger convex} (or \emph{M-convex}) Menger \cite{MR1512479} if,
for every $x,y\in X$ with $x\neq y$, there exists a point
$z\in X\setminus\{x,y\}$ with $d(x,y)=d(x,z)+ d(z,y)$. In this case,
$z$ is said to \emph{lie between $x$ and $y$}.  The relation with
convexity is that any two points of a complete Menger convex space can
be joined by a metric segment. Namely, the following theorem was
proved by Menger \cite{MR1512479}.

\begin{theorem}
  \label{theorem-Lipschitz-ext:15}
  If\/ $(X,d)$ is Menger convex and complete, $x,y\in X$ and
  $\ell=d(x,y)$, then there exists an isometry\/
  $\gamma:[0,\ell]\to X$ such that $\gamma(0)=x$ and $\gamma(\ell)=y$.
\end{theorem}

Menger's original proof was based on a transfinite induction.  An
alternative proof avoiding transfinite induction was given in
Blumenthal \cite[Theorem 14.1]{MR0268781} and credited to
Aronszajn.\medskip

For a metric space $(X,d)$ and $x,y\in X$, an isometry
$\gamma:[0,\ell]\to X$ with $\gamma(0)=x$ and $\gamma(\ell)=y$ is
usually called a \emph{geodesic line} (or \emph{geodesic}) from $x$ to
$y$. Evidently, such an isometry can be represented as the path
$\alpha:[0,1]\to X$ such that $\alpha(0)=x$, $\alpha(1)=y$ and
$d(\alpha(s),\alpha(t))=d(x,y)|s-t|$ for every $s,t\in [0,1]$. In
particular, for every $\lambda\in [0,1]$, the point
$z=\alpha(\lambda)$ has the property that $d(x,z)=\lambda d(x,y)$ and
$d(z,y)=(1-\lambda) d(x,y)$. This property reflects another
interpretation of convexity in metric spaces. Namely, a metric space
$(X,d)$ is called \emph{metrically convex}, see \cite{MR0461107}, if
for every $x,y\in X$ and $0<\lambda<1$, there exists a point $z\in X$
with $d(x,z)=\lambda d(x,y)$ and $d(z,y)=(1-\lambda)
d(x,y)$. Equivalently, $(X,d)$ is metrically convex if for every
$s,t\geq 0$ and points $x,y\in X$ with $d(x,y)\leq s+t$, there exists
a point $z\in X$ with $d(x,z)\leq s$ and $d(z,y)\leq t$. Let us
explicitly remark that in some sources, metrically convex is used for
Menger convex, while in other sources the above property is used to
define Menger convexity. According to \Cref{theorem-Lipschitz-ext:15},
both properties are equivalent for complete metric spaces.\medskip

The following stronger property of the least modulus of continuity is
valid in the setting of a metrically convex domain.

\begin{proposition}
  \label{proposition-Lipschitz-ext:10}
  Let $(X,d)$ and $(Y,\rho)$ be metric spaces. If\/ $(X,d)$ is
  metrically convex and $f:X\to Y$, then the least modulus of
  continuity $\omega_f$ for $f$ is subadditive.
\end{proposition}

\begin{proof}
  Take $s,t\geq 0$ and points $x_s,x_t\in X$ with
  $d(x_s,x_t)\leq s+t$. Since $(X,d)$ is metrically convex, there
  exists $z\in X$ such that $d(x_s,z)\leq s$ and $d(z,x_t)\leq
  t$. According to \cref{eq:Lipschitz-ext:16}, this implies that
  $\omega_f(s+t)\leq \omega_f(s)+\omega_f(t)$ because
  \[
    \rho(f(x_s),f(x_t))\leq \rho(f(x_s),f(z))+\rho(f(z),f(x_t))\leq
    \omega_f(s) +\omega_f(t).\qedhere
  \]
\end{proof}

The above property implies the following consequence of
\Cref{theorem-Lipschitz-ext:8}, which is complementary to
\Cref{corollary-KIP-ver-21:1}. In this consequence, a subset
$A\subseteq X$ of a metric space $(X,d)$ is \emph{metrically convex}
if it is so as a subspace of $(X,d)$, i.e.\ with respect to the metric
$d\uhr A\times A$.

\begin{corollary}
  \label{corollary-Lipschitz-ext:10}
  If\/ $(X,d)$ is a metric space and $A\subseteq X$ is a metrically
  convex subset, then each uniformly continuous function
  $\varphi:A\to \R$ can be extended to a uniformly continuous function
  $f:X\to \R$.
\end{corollary}

\begin{proof}
  Let $\varphi:A\to \R$ be uniformly continuous and
  $\omega=\omega_\varphi$ be the least modulus of continuity for
  $\varphi$. Then $\lim_{t\to 0}\omega(t)=0$ because $\varphi$ is
  uniformly continuous. Moreover, by
  \Cref{proposition-Lipschitz-ext:10}, $\omega$ is also subadditive
  because $A$ is metrically convex. Thus, by
  \Cref{theorem-Lipschitz-ext:8}, $\varphi$ can be extended to a
  function $f:X\to \R$ preserving the modulus of continuity $\omega$,
  i.e.\ to a uniformly continuous function.
\end{proof}

\Cref{proposition-Lipschitz-ext:10} is also very useful in
constructing counterexamples regarding the subtle difference between
\Cref{corollary-KIP-ver-21:1} and \Cref{corollary-Lipschitz-ext:10}.

\begin{example}
  \label{example-Lipschitz-ext:4}
  Define $\varphi:\N\to \R$ by $\varphi(n)=n^2$ for every
  $n\in\N$. Evidently, $\varphi$ is uniformly continuous. In fact, one
  can easily see that the least modulus of continuity $\omega_\varphi$
  for $\varphi$ is the extended function
  \begin{equation}
    \label{eq:Lipschitz-ext:22}
    \omega_\varphi(t)=
    \begin{cases}
      0 &\text{if $0\leq t<1$, and}\\
      +\infty &\text{if $1\leq t$.}
    \end{cases}
  \end{equation}
  However, $\varphi$ cannot be extended to a uniformly continuous
  function on $\R$. Indeed, suppose that $f:\R\to \R$ is a uniformly
  continuous function with $f\uhr \N=\varphi$. Then by
  \Cref{proposition-Lipschitz-ext:10}, the least modulus of continuity
  $\omega_f$ for $f$ is a subadditive function with
  $\lim_{t\to 0}\omega_f(t)=0$. Since $\omega_f$ is also increasing,
  \Cref{proposition-Uniform-Cont-Ext:3} implies that
  ${\omega_f(t)<+\infty}$ for every $t\geq 0$. However, we also have
  that ${\omega_\varphi\leq \omega_f}$ because ${f\uhr
    \N=\varphi}$. According to \cref{eq:Lipschitz-ext:22}, this is
  impossible.\qed
\end{example}

McShane's extension theorem for uniformly continuous functions
(\Cref{theorem-Lipschitz-ext:8}) was also generalised to Hilbert
spaces. The following theorem was obtained by Gr\"unbaum and
Zarantonello in \cite[Theorem 2]{Gruenbaum1968}.

\begin{theorem}
  \label{theorem-Uni-Ball-Intersc:1}
  Let $H$ be a Hilbert space, $L$ be an inner product space and
  $A\subseteq L$. Then a uniformly continuous map $\varphi:A\to H$ has
  a uniformly continuous extension $f:L\to H$ if and only if\/
  $\varphi$ has a subadditive modulus of uniform continuity.
\end{theorem}

In the above theorem, taking $A\subseteq L$ to be convex, it follows
from \Cref{proposition-Lipschitz-ext:10} that the least modulus of
continuity $\omega_\varphi$ of a uniformly continuous map
${\varphi:A\to H}$ is subadditive and
$\lim_{t\to 0}\omega_\varphi(t)=0$. Hence,
\Cref{theorem-Uni-Ball-Intersc:1} implies the following immediate
consequence, it was obtained in \cite[Corollary]{Gruenbaum1968}.

\begin{corollary}
  \label{corollary-Uni-Ball-Intersc:1}
  Let $H$ be a Hilbert space, $L$ be an inner product space and
  $A\subseteq L$ be a convex subset. Then each uniformly continuous
  map $\varphi:A\to H$ has a uniformly continuous extension
  $f:L\to H$.
\end{corollary}

In fact, in the one direction, \Cref{theorem-Uni-Ball-Intersc:1} is
also based on \Cref{proposition-Lipschitz-ext:10}. Namely, if
$f:L\to H$ is a uniformly continuous extension of $\varphi:A\to H$,
then by this proposition, the least modulus of continuity $\omega_f$
of $f$ is subadditive and $\lim_{t\to 0}\omega_f(t)=0$. Moreover,
$\omega_f$ is also a modulus of continuity for $\varphi$ because
$\varphi=f\uhr A$.\medskip

In the other direction, the proof is based on the fact that each
increasing subadditive function $\eta:[0,+\infty)\to [0,+\infty)$ has
a concave majorant. For such a function $\eta$, it was shown in
\cite[Lemma 5]{Gruenbaum1968}, see also \cite{Gruenbaum1966}, that
$\eta\leq \omega\leq 2\eta$ for some concave function
$\omega:[0,+\infty)\to [0,+\infty)$. Below, we provide a simplified
version of this fact based on the following observation in the proof of
\cite[Theorem 1]{MR84762}.

\begin{proposition}
  \label{proposition-Uni-Ball-Intersc:1}
  If $\eta:[0,+\infty)\to [0,+\infty)$ is an increasing subadditive
  function, then
  \begin{equation}
    \label{eq:Uniform-Cont-Ext:3}
    \varlimsup_{t\to +\infty}\mfrac{\eta(t)}t\leq \eta(1).
  \end{equation}
  In particular, such a function $\eta$ admits a concave function
  $\omega:[0,+\infty)\to [0,+\infty)$ with\/ $\eta\leq \omega$ and\,
  $\lim_{t\to 0}\omega(t)=\lim_{t\to 0}\eta(t)$.
\end{proposition}

\begin{proof}
  For each $t> 0$, let $\lfloor t\rfloor$ be the greatest integer that
  is less than or equal to $t$. The resulting function
  $\lfloor\cdot\rfloor:(0,+\infty)\to \Z$ is commonly called the
  \emph{floor function}, or the \emph{greatest integer
    function}. Since $t\leq \lfloor t\rfloor +1$ and $\omega$ is
  increasing and subadditive, it follows that
  $\frac{\omega(t)}t \leq \frac{\omega\left(\lfloor t\rfloor
      +1\right)}t\leq \frac{\lfloor t\rfloor +1}t \cdot
  \omega(1)$. Accordingly,
  \[
    \varlimsup_{t\to +\infty} \mfrac{\omega(t)}t\leq \varlimsup_{t\to
      +\infty} \mfrac{\lfloor t\rfloor +1}t \cdot \omega(1)=\omega(1).
  \]
  The rest of the proof follows from
  \Cref{theorem-Uniform-Cont-Ext:1}.
\end{proof}

Thus, in our setting, if $\varphi:A\to H$ has a subadditive modulus of
uniform continuity $\eta$, then by
\Cref{proposition-Uni-Ball-Intersc:1}, $\varphi$ also has a concave
modulus of uniform continuity $\omega$. For such a concave function
$\omega:[0,+\infty)\to [0,+\infty)$, Gr\"unbaum and Zarantonello
\cite{Gruenbaum1968} showed the following property which constitutes
the most difficult part of the proof of
\Cref{theorem-Lipschitz-ext:8}.

\begin{theorem}
  \label{theorem-Uni-Ball-Intersc:2}
  Let $\omega:[0,+\infty)\to [0,+\infty)$ be a concave function with
  $\omega(0)=0$. Then for arbitrary vectors\,
  $\mathbf{u}_1,\dots, \mathbf{u}_k\in \R^n$ and\,
  $\lambda_1,\dots, \lambda_k\geq 0$,
  \begin{equation}
    \label{eq:Kirszbraun-Intersc-v12:1}
    \sum_{i,j=1}^k\left[\omega^2\left(\|\mathbf{u}_i\|\right)
      +\omega^2\left(\|\mathbf{u}_j\|\right)- 
      \omega^2\left(\|\mathbf{u}_i-\mathbf{u}_j\|\right)\right]
    \lambda_i\lambda_j 
    \geq 0.
  \end{equation}
\end{theorem}

The inequality in \cref{eq:Kirszbraun-Intersc-v12:1} is actually
identical to the Gr\"unbaum-Zarantonello condition in
\cref{eq:KIP-ver-17:1}, which shows that
\Cref{theorem-Uni-Ball-Intersc:2} provides a direct reduction of
\Cref{theorem-Uni-Ball-Intersc:1} to the general result in
\Cref{theorem-KIP-ver-17:1}.

\begin{proof}[Proof that \Cref{theorem-Uni-Ball-Intersc:2}
  implies \Cref{theorem-Uni-Ball-Intersc:1}]
  Assume that $H$, $L$ and $A\subseteq L$ are as in
  \Cref{theorem-Uni-Ball-Intersc:1}, and $\varphi:A\to H$ is a
  uniformly continuous map that has a subadditive modulus of uniform
  continuity $\eta:[0,+\infty)\to [0,+\infty)$. Then, according to
  \Cref{proposition-Uni-Ball-Intersc:1}, $\varphi$ also has a concave
  modulus of uniform continuity $\omega:[0,+\infty)\to
  [0,+\infty)$. The nontrivial case is when $\omega$ is not a constant
  function. In this case, by
  \Cref{proposition-Lipschitz-ext:4,proposition-Uniform-Cont-Ext:3},
  $\omega(t)>0$ for every $t>0$. Therefore, we may define a metric
  $\rho$ on $L$ by $\rho_{xy}=\rho(x,y)= \omega(\|x-y\|)$ for every
  $x,y\in L$. Next, take a nonempty finite set $T\subseteq L$ and a
  point $p\in L\setminus T$. It now follows from
  \cref{eq:Kirszbraun-Intersc-v12:1} that the quadratic form
  $Q_{(p,T)}$ satisfies the condition in \cref{eq:KIP-ver-17:1} 
  because $\xi\in \R^T_+$ implies that
  \begin{align*}
    Q_{(p,T)}(\xi)&=\sum_{(s,t)\in T^2}\left[\rho_{ps}^2+
                    \rho_{pt}^2-
                    \rho_{st}^2\right]\xi(s)\xi(t) \\
                  &=\sum_{(s,t)\in T^2}\left[\omega^2(\|s-p\|)+
                    \omega^2(\|t-p\|)-
                    \omega^2(\|s-t\|)\right]\xi(s)\xi(t)\geq 0.
  \end{align*}
  Furthermore, $\varphi:A\to H$ is a nonexpansive map with respect to
  the metric $\rho$ on $L$. Hence, by \Cref{theorem-KIP-ver-17:1},
  $\varphi$ can be extended to a nonexpansive map $f:L\to H$ with
  respect to this metric $\rho$ on $L$. This is equivalent to the fact
  that $\omega$ is also a modulus of continuity for $f:L\to H$, i.e.\
  that $f$ is a uniformly continuous extension of $\varphi$.
\end{proof}

The proof of \Cref{theorem-Uni-Ball-Intersc:2} in \cite{Gruenbaum1968}
takes up most of this paper and is based on several approximation
results of concave functions with twice differentiable such
functions. The theory of mollifiers, introduced by Sobolev in
\cite{Soboleff1938} and Friedrichs in \cite{Friedrichs1944}, plays a
significant role in this. Essentially the same proof is reproduced in
\cite{MR0461107} and in less detail also in
\cite{Benyamini2000}. Below we will present a very simple and
self-contained proof of \Cref{theorem-Uni-Ball-Intersc:2}, based only
on the elementary concepts used in \Cref{theorem-Rem-Lipschitz:1}. A
key role in it will be played by the property of concave functions
${\eta:[0,+\infty)\to [0,+\infty)}$ that for arbitrary vectors
$\mathbf{u}_1, \dots, \mathbf{u}_k \in \mathbb{R}^n$ and
$\lambda_1, \dots, \lambda_k \geq 0$, the following inequality holds.
\begin{equation}
  \label{eq:Kirszbraun-Intersc-v6:1}
  \sum_{i,j=1}^k
  \left[\eta\left(\|\mathbf{u}_i\|^2\right) 
    +\eta\left(\|\mathbf{u}_j\|^2\right) -
    \eta\left(\|\mathbf{u}_i-\mathbf{u}_j\|^2\right) 
  \right] \lambda_i\lambda_j. 
\end{equation}
In this inequality, if $\mathbf{u}_i\in \R^n$ is a zero vector for
some $i\leq k$, then
\[
  \eta\left(\|\mathbf{u}_i\|^2\right)
  +\eta\left(\|\mathbf{u}_j\|^2\right) -
  \eta\left(\|\mathbf{u}_i-\mathbf{u}_j\|^2\right) =
  \eta(0)+\eta\left(\|\mathbf{u}_j\|^2\right)
  -\eta\left(\|\mathbf{u}_j\|^2\right)= \eta(0)\geq0.
\]
Hence, \cref{eq:Kirszbraun-Intersc-v6:1} is essentially reduced to
nonzero vectors $\mathbf{u}_1, \dots, \mathbf{u}_k \in
\mathbb{R}^n$.\medskip  

To establish this property of concave functions, we need two simple
observations, the first of which essentially coincides with
\cite[Lemma 3]{Gruenbaum1968}.

\begin{proposition}
  \label{proposition-KIP-draft-05:1}
  Let $\alpha:[0,+\infty)\to [0,+\infty)$ be a function which
  satisfies \cref{eq:Kirszbraun-Intersc-v6:1} and ${r>0}$. Then the
  pointwise minimum\, $\alpha_r=\alpha\wedge r$ also satisfies
  \cref{eq:Kirszbraun-Intersc-v6:1}.
\end{proposition}

\begin{proof}
  Let $\mathbf{u}_1, \dots, \mathbf{u}_k \in \mathbb{R}^n$ and
  $\lambda_1, \dots, \lambda_k \geq 0$. If
  $\alpha\left(\|\mathbf{u}_j\|^2\right)\geq r$ for some $j\leq k$,
  then for every $i\leq k$ we have that
  $\alpha_r\left(\|\mathbf{u}_j\|^2\right)= r\geq
  \alpha_r\left(\|\mathbf{u}_i-\mathbf{u}_j\|^2\right)$. Accordingly,
  in this case,
  \[
    \sum_{i=1}^k \left[\alpha_r\left(\|\mathbf{u}_i\|^2\right)
    +\alpha_r\left(\|\mathbf{u}_j\|^2\right)
    -\alpha_r\left(\|\mathbf{u}_i-\mathbf{u}_j\|^2\right)\right] \geq
  \sum_{i=1}^k \alpha_r\left(\|\mathbf{u}_i\|^2\right)\geq 0. 
  \]
  Regarding the rest of the terms of the sum in
  \cref{eq:Kirszbraun-Intersc-v6:1}, let
  $K=\left\{i\leq k: \alpha\left(\|\mathbf{u}_i\|^2\right) \leq
    r\right\}$. If $K\neq \emptyset$, since $\alpha$ satisfies
  \cref{eq:Kirszbraun-Intersc-v6:1} and
  $\alpha_r\left(\|\mathbf{u}_i-\mathbf{u}_j\|^2\right)\leq
  \alpha\left(\|\mathbf{u}_i-\mathbf{u}_j\|^2\right)$ for every
  $i,j\leq k$, we obtain that
  \begin{align*}
    \sum_{(i,j)\in K^2} \big[\alpha_r\left(\|\mathbf{u}_i\|^2\right)
    +\alpha_ r\left(\|\mathbf{u}_j\|^2\right)-
    \alpha_ r\left(\|\mathbf{u}_i-\mathbf{u}_j\|^2\right)
    \big]\lambda_i\lambda_j&\geq\\
    \sum_{(i,j)\in K^2} \big[\alpha\left(\|\mathbf{u}_i\|^2\right)
    +\alpha\left(\|\mathbf{u}_j\|^2\right)-
    \alpha\left(\|\mathbf{u}_i-\mathbf{u}_j\|^2\right)
    \big]\lambda_i\lambda_j&\geq 0.\qedhere
  \end{align*}
\end{proof}

The second observation follows an idea in the proof of \cite[Lemma
1.14]{Benyamini2000}.

\begin{proposition}
  \label{proposition-KIP-vgg-draft-07:1}
  Let $a_1\geq \dots \geq a_{k+1}\geq 0$, $b_1,\dots, b_{k+1}\geq 0$
  and\/ $0<r_1\leq\dots\leq r_k$ be such that
  \begin{equation}
    \label{eq:KIP-VGG-draft:1}
    a_ir_i+b_i= a_{i+1} r_i+ b_{i+1}\quad \text{for every $i\leq k$.}
  \end{equation}
  Then the function $\gamma(t)=\bigwedge_{i=1}^{k+1}(a_it+b_i)$ for
  $t\geq 0$, i.e.\ the pointwise minimum of these affine functions on
  the interval\/ $[0, +\infty)$, satisfies
  \cref{eq:Kirszbraun-Intersc-v6:1}.
\end{proposition}

\begin{proof}
  For $1\leq i\leq k$, by condition, the affine function
  $(a_i-a_{i+1})t+ (b_i-b_{i+1})$ is increasing. Moreover, by
  \cref{eq:KIP-VGG-draft:1}, this function is zero at
  $t=r_i$. Therefore, we have that $a_{i}t+b_i\leq a_{i+1}t+b_{i+1}$
  for all $0\leq t\leq r_i$ and $a_{i+1}t+b_{i+1}\leq a_{i}t+b_i$ for
  all $t\geq r_i$. Thus, for $t\geq 0$, the function
  ${\gamma(t)=\bigwedge_{i=1}^{k+1}(a_it+b_i)}$ is simply given by
  \begin{equation}
    \label{eq:KIP-VGG-draft:3}
    \gamma(t)=
    \begin{cases}
      a_1t+b_1 &\text{if $0\leq t\leq r_1$,} \\
      a_{i}t+b_{i} &\text{if
        $r_{i-1}\leq t\leq r_{i}$ for some $1< i\leq k$, and} \\
      a_{k+1}t+b_{k+1} &\text{if $t\geq r_{k+1}$.}
    \end{cases}
  \end{equation}
  In order to show that $\gamma(t)$ satisfies
  \cref{eq:Kirszbraun-Intersc-v6:1}, we will use that the identity
  function $t:[0,+\infty)\to [0,+\infty)$ clearly satisfies the
  inequality in \cref{eq:Kirszbraun-Intersc-v6:1}. Hence, by
  \Cref{proposition-KIP-draft-05:1}, so does each pointwise minimum
  $t\wedge r_i$ for $1\leq i\leq k$. Accordingly, since $b_1\geq 0$
  and $a_1\geq \dots \geq a_{k+1}\geq 0$, the function
  \[
    \beta(t)=b_1+\sum_{i=1}^k \left(a_{i}-a_{i+1}\right) [t\wedge
    r_i] +a_{k+1}t,\quad t\geq 0,
  \]
  also satisfies \cref{eq:Kirszbraun-Intersc-v6:1}. The proof now
  reduces to showing that $\gamma = \beta$, which relies on the
  following representation of the ``residual parts'' of the sum by
  which $\beta(t)$ is defined and on the property from
  \cref{eq:KIP-VGG-draft:1}. Namely, for every ${1\leq i\leq k}$ and
  $0\leq t\leq r_i$, we have that
  \begin{equation}
    \label{eq:KIP-VGG-draft:2}
\left\{
      \begin{aligned}
        a_it&=\medop\sum_{j=i}^{k}\left(a_{j}-a_{j+1}\right)
        \left[t\wedge r_j\right] +a_{k+1}t,\quad
        \text{
          and} \\[-5pt]
        b_{i+1}&=b_1+\medop\sum_{j=1}^i\left(a_j-a_{j+1}\right)r_j.
    \end{aligned}
  \right.
\end{equation}
The first equality in \cref{eq:KIP-VGG-draft:2} follows easily from
the fact that $t\wedge r_j=t$ for every ${t\leq r_i\leq r_j}$. The
second equality is an immediate consequence of
\cref{eq:KIP-VGG-draft:1} because
\[
  b_1+\sum_{j=1}^i\left(a_j-a_{j+1}\right)r_j= \sum_{j=1}^i
  \left[\left(a_jr_j+b_j\right)
    -\left(a_{j+1}r_j+b_{j+1}\right)\right] +b_{i+1}=b_{i+1}.
\]
    
The representations in \cref{eq:KIP-VGG-draft:3,eq:KIP-VGG-draft:2}
directly show that $\beta(t)=b_1 +a_1t=\gamma(t)$ for
$0\leq t\leq r_1$.  Similarly, for the same reason,
$\beta(t)= a_{i}t+b_{i}= \gamma(t)$ for $1<i\leq k$ and
$r_{i-1}\leq t\leq r_{i}$ because
  \[
    \beta(t)=b_1+\sum_{j=1}^{i-1} \left(a_{j}-a_{j+1}\right)
    r_j +\sum_{j=i}^{k}\left(a_{j}-a_{j+1}\right)
        \left[t\wedge r_j\right] +a_{k+1}t=
    b_{i}+a_{i}t=\gamma(t). 
  \]
  Finally, \cref{eq:KIP-VGG-draft:3,eq:KIP-VGG-draft:2} also imply
  that $\gamma(t)=\beta(t)$ for $t\geq r_k$ because
  \[
    \beta(t)=b_1+\sum_{i=1}^k \left(a_{i}-a_{i+1}\right)
    r_i +a_{k+1} t=
    b_{k+1}+a_{k+1}t=\gamma(t).\qedhere 
  \]
\end{proof}
  
\begin{proposition}
  \label{proposition-KIP-ver-21:1}
  If\/ $\eta:[0,+\infty)\to [0,+\infty)$ is a concave function with
  $\eta(0)=0$, then it satisfies \cref{eq:Kirszbraun-Intersc-v6:1}.
\end{proposition}

\begin{proof}
  Obviously, the property is trivial for a constant
  function. Otherwise, by
  \Cref{proposition-Lipschitz-ext:4,proposition-Uniform-Cont-Ext:1,proposition-Uniform-Cont-Ext:3},
  we get that $\eta(t)>0$ for every $t>0$. In this case, the
  inequality in \cref{eq:Kirszbraun-Intersc-v6:1} follows from the
  representation in \cref{eq:Rem-Lipschitz:1}. Namely, according to
  \Cref{theorem-Rem-Lipschitz:1}, there exists a family $\mathcal{A}$
  of affine functions such that
  $\eta(t)=\min_{\alpha\in \mathcal{A}}\alpha(t)$ for every
  $t>0$. Therefore, if $\alpha\in \mathcal{A}$, then $\alpha(t)=at +b$
  for some $a,b\geq 0$ because $\alpha(t)>0$ for every $t>0$. Take now
  vectors
  $\mathbf{u}_1, \dots, \mathbf{u}_k \in
  \mathbb{R}^n\setminus\{\mathbf{0}\}$ and
  $\lambda_1, \dots, \lambda_k \geq 0$, and next set
  $ S=\left\{\|\mathbf{u}\|_1^2,\dots, \|\mathbf{u}\|_k^2\right\}$.
  Then $S\subseteq (0,+\infty)$, and by the property of the family
  $\mathcal{A}$ in \cref{eq:Rem-Lipschitz:1}, for every $s\in S$ there
  exists $\alpha_s\in \mathcal{A}$ such that $\alpha_s(s)=\eta(s)$. If
  $s, t \in S$ are such that $s < t$ and there is no other element of
  $S$ between $s$ and $t$, then $\alpha_s(s) \leq \alpha_t(s)$ and
  $\alpha_t(t) \leq \alpha_s(t)$. Hence, since the functions
  $\alpha_s, \alpha_t \in \mathcal{A}$ represent straight lines, there
  exists $s\leq r\leq t$ such that $\alpha_s(x) \leq \alpha_t(x)$ for
  all $x \leq r$ and $\alpha_t(x) \leq \alpha_s(x)$ for all
  $x \geq r$.
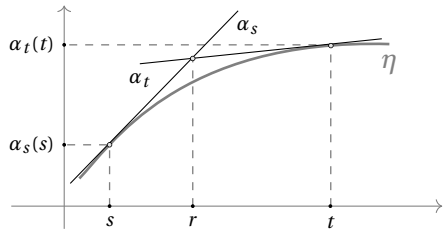
  \begin{center}
    \begin{tikzpicture}
      \draw[white!50!black,->] (-.7,0) -- (5,0);
      \draw[white!50!black,->] (0,-.3) -- (0,2.65);
      \draw[gray,dashed,line width=0.5pt,-] (.6,0) -- (.6,.81) --
      (0,.81); \draw[gray,dashed,line width=0.5pt,-] (3.525,0) --
      (3.525,2.13) -- (0,2.13); \draw[gray,dashed,line width=0.5pt,-]
      (1.7,0) -- (1.7,1.95); \draw[black!50!white,line width=1pt]
      (.2,.365) .. controls (.6,.78) and (1.5,2.22) .. (4.3,2.14)
      (4.3,2.14) node[below] {\small $\eta$}; \filldraw [black]
      (3.527,2.12) circle [radius=.8pt] (0,2.13) circle [radius=.5pt]
      node[left] {\tiny $\alpha_t(t)$}; \filldraw [black] (.6,.81)
      circle [radius=.8pt] (0,.81) circle [radius=.5pt] node[left]
      {\tiny $\alpha_s(s)$}; \filldraw [black] (1.7,1.95) circle
      [radius=.8pt] ; \filldraw [black] (.6,0) circle [black,
      radius=.5pt] node[below] {\tiny $s$} (1.7,0) circle [black,
      radius=.5pt] node[below] {\tiny $r$} (3.52,0) circle [black,
      radius=.5pt] node[below] {\tiny $t$}; \draw[black,line
      width=0.4pt,-] (.08,.3) -- (2.3,2.58) node[below] {\tiny
        \quad$\alpha_s$}; \draw[black,line width=0.4pt,-] (1,1.88) --
      (4.2,2.21) (1,1.88) node[below] {\tiny $\alpha_t$}; \filldraw
      [white!90!black] (3.527,2.12) circle [radius=.5pt] (.6,.81)
      circle [radius=.5pt] (1.7,1.95) circle [radius=.5pt];
    \end{tikzpicture}
  \end{center}
  In other words, the slope of $\alpha_t(x)$ does not exceed that of
  $\alpha_s(x)$. Accordingly, by
  \Cref{proposition-KIP-vgg-draft-07:1}, the pointwise minimum
  $\gamma(x)=\bigwedge_{s\in S}\alpha_s(x)$, $x\geq 0$, of these
  affine functions on the interval $[0, +\infty)$ satisfies
  \cref{eq:Kirszbraun-Intersc-v6:1}, and therefore
  \[
    \sum_{i,j=1}^k \left[\gamma\left(\|\mathbf{u}_i\|^2\right)
      +\gamma\left(\|\mathbf{u}_j\|^2\right)-
      \gamma\left(\|\mathbf{u}_i-\mathbf{u}_j\|^2\right)
    \right]\lambda_i\lambda_j\geq 0.
  \]
  Since
  $\eta\left(\|\mathbf{u}_i\|^2\right)=\gamma\left(\|\mathbf{u}_i\|^2\right)$
  for $i\leq k$, and
  $\eta\left(\|\mathbf{u}_i-\mathbf{u}_j\|^2\right)\leq
  \gamma\left(\|\mathbf{u}_i-\mathbf{u}_j\|^2\right)$ for $i,j\leq k$,
  this shows that we also have
  $\sum_{i,j=1}^k \left[\eta\left(\|\mathbf{u}_i\|^2\right)
    +\eta\left(\|\mathbf{u}_j\|^2\right)-
    \eta\left(\|\mathbf{u}_i-\mathbf{u}_j\|^2\right)
  \right]\lambda_i\lambda_j\geq 0$.
\end{proof}

Based on \Cref{proposition-KIP-ver-21:1}, we can apply
\Cref{theorem-Rem-Lipschitz:1} once more to complete the proof of
\Cref{theorem-Uni-Ball-Intersc:2}.

\begin{proof}[Proof of \Cref{theorem-Uni-Ball-Intersc:2}]
  Let $\omega:[0,+\infty)\to [0,+\infty)$ be as in
  \Cref{theorem-Uni-Ball-Intersc:2}, and let $\mathcal{A}$ be the
  corresponding family of affine functions, as in
  \cref{eq:Rem-Lipschitz:1}, applied to this concave function. If
  $\alpha\in \mathcal{A}$, then $\alpha(t)=at +b$ for some $a,b\geq 0$
  because $\alpha\geq \omega$. Accordingly, to each
  $\alpha(t)=at+b\in \mathcal{A}$ we can associate the concave
  function ${\beta_\alpha:(0,+\infty)\to [0,+\infty)}$ defined by
  $\beta_\alpha(t)=a^2t+ 2ab\sqrt{t}+ b^2=
  \alpha^2\left(\sqrt{t}\right)$. Finally, we can define another
  function $\eta:[0,+\infty)\to [0,+\infty)$ by $\eta(0)=0$ and
  $\eta(t)=\inf_{\alpha\in \mathcal{A}}\beta_\alpha(t)$ for every
  $t>0$. Evidently, $\eta\uhr(0,+\infty)$ is concave as an infimum of
  concave functions. Therefore, by
  \Cref{proposition-Uniform-Cont-Ext:8}, it is itself concave. Thus,
  according to \Cref{proposition-KIP-ver-21:1}, $\eta$ satisfies
  \cref{eq:Kirszbraun-Intersc-v6:1}. However,
  $\eta(t)=\omega^2\left(\sqrt{t}\right)$ for every $t\geq 0$ because
  $\omega(t)=\min_{\alpha\in \mathcal{A}}\alpha(t)$ and
  $\eta(t)=\inf_{\alpha\in \mathcal{A}}\beta_\alpha(t)=
  \min_{\alpha\in \mathcal{A}}\alpha^2\left(\sqrt{t}\right)$ for
  $t>0$. In other words, this means that the function $\omega(t)$
  satisfies the inequality in \cref{eq:Kirszbraun-Intersc-v12:1}, and
  the proof is complete.
\end{proof}

In conclusion, let us remark that the function $\omega(t)=t^\alpha$ is
concave for $0<\alpha\leq1$. Therefore, in this case, by taking
$\omega(t)=t^\alpha$ in \Cref{theorem-Uni-Ball-Intersc:2}, we obtain
an alternative proof of the inequality in
\cref{eq:Kirszbraun-Intersc-v3:4}, and thus of
\Cref{theorem-Kirszbraun-Intersc-v4:2} as well.


\end{document}